\documentclass[12pt,twoside]{amsart}
\usepackage[T1]{fontenc}
\usepackage[utf8]{inputenc}
\usepackage[british]{babel}
\usepackage{amsmath,amssymb,amsthm,amscd}
\usepackage{enumerate}
\usepackage{verbatim}
\usepackage{lmodern}
\usepackage{color}
\usepackage{xcolor}
\usepackage{faktor}
\usepackage[left=2.5cm,right=2.5cm, top=2.8cm,bottom=2.9cm]{geometry}
\theoremstyle{definition}

\newtheorem{thm}{Theorem}[section]
\newtheorem{dfn}[thm]{Definition}
\newtheorem{lem}[thm]{Lemma}

\newtheorem{rem}[thm]{Remark}

   \newcommand{\Cdb}{\mbox{$\mathbb{C}$}}

   \newcommand{\Ndb}{\mbox{$\mathbb{N}$}}

   \newcommand{\Rdb}{\mbox{$\mathbb{R}$}}
   
   \newcommand{\Tdb}{\mbox{$\mathbb{T}$}}

    \newcommand{\M}{\mbox{${\mathcal M}$}}

\newcommand{\norm}[1]{\Vert#1\Vert}
\newcommand{\bignorm}[1]{\bigl\Vert#1\bigr\Vert}
\newcommand{\Bignorm}[1]{\Bigl\Vert#1\Bigr\Vert}

\author[L. Arnold]{Loris Arnold}
\email{larnold@impan.pl}
\address{Institute of Mathematics, Polish Academy of Sciences, Śniadeckich 8, Warszawa, Poland}
\author[C. Le Merdy]{Christian Le Merdy}
\email{clemerdy@univ-fcomte.fr}
\address{Universit\'e de Franche-Comt\'e, Laboratoire de Math\'ematiques de Besan\c con, UMR CNRS 6623,
16 Route de Gray, 25000 Besan\c{c}on, France}
\author[S. Zadeh]{Safoura Zadeh}
\email{jsafoora@gmail.com}
\address{Laboratoire de Math\'ematiques Blaise Pascal, UMR 6620, Universit\'e Clermont Auvergne, France}
\address{School of Mathematics, University of Bristol,
Bristol BS8 1UG, United Kingdom}
\address{Institut des Hautes \'Etudes Scientifiques,
Bures-sur-Yvette 91440, France}

\title[Hankel operators on $L^p(\Rdb_+)$]{Hankel 
operators on $L^p(\Rdb_+)$ and their $p$-completely bounded
multipliers}

\begin{document}
\maketitle

\begin{abstract}
We show that for any $1<p<\infty$, the space $Hank_p(\mathbb{R}_+)\subseteq B(L^p(\Rdb_+))$ 
of all Hankel operators on $L^p(\Rdb_+)$ is equal to the 
$w^*$-closure of the linear span of the operators
$\theta_u\colon L^p(\Rdb_+)\to L^p(\Rdb_+)$ defined by 
$\theta_uf=f(u-\,\cdotp)$, for $u>0$. We deduce that 
$Hank_p(\mathbb{R}_+)$ is the dual space of
$A_p(\Rdb_+)$, a half-line analogue of the 
Figa-Talamenca-Herz algebra
$A_p(\Rdb)$. Then we show that 
a function $m\colon \Rdb_+^*\to\Cdb$ is the 
symbol of a $p$-completely bounded
multiplier $Hank_p(\mathbb{R}_+)\to Hank_p(\mathbb{R}_+)$
if and only if there exist 
$\alpha\in L^\infty(\Rdb_+;L^p(\Omega))$ and
$\beta\in L^\infty(\Rdb_+;L^{p'}(\Omega))$ such that
$m(s+t)=\langle\alpha(s),\beta(t)\rangle$
for a.e. $(s,t)\in\Rdb_+^{*2}$. We also 
give analogues of these results in the (easier) discrete case.
\end{abstract}

\section{Introduction} \label{sec:introduction}
For any $u>0$ and for any function $f\colon \Rdb_+\to\Cdb$,
let $\tau_uf\colon \Rdb_+\to\Cdb$ be the shifted 
function defined by 
$\tau_u f =f(\cdotp -u)$. Let $1<p,p'<\infty$ be two 
conjugate indices. We say that 
a bounded operator $T\colon L^p(\Rdb_+)\to L^p(\Rdb_+)$ 
is Hankelian
if $\langle T\tau_u f,g\rangle= \langle T f,\tau_u g\rangle$
for all $f\in L^p(\Rdb_+)$ and $g\in L^{p'}(\Rdb_+)$.
Let $B(L^p(\Rdb_+))$ denote the Banach space of all bounded
operators on $L^p(\Rdb_+)$.
The main object of this paper is the subspace
$Hank_p(\mathbb{R}_+)\subseteq B(L^p(\Rdb_+))$ 
of all Hankel operators on $L^p(\Rdb_+)$.

The case $p=2$ has received a lot of attention, see
\cite{Ni0, Ni, Peller,Y0,Y1,Y2} and the references therein.
The most important result in this case
is that $Hank_2(\mathbb{R}_+)$ 
is isometrically
isomorphic to the quotient space 
$\frac{L^\infty({\mathbb R})}{H^\infty({\mathbb R})}$, where 
$H^\infty(\Rdb)\subset L^\infty(\Rdb)$ is the classical
Hardy space of essentially bounded functions whose Fourier 
transform has support in $\Rdb_+$
(see \cite[Section IV.5.3]{Ni} or \cite[Theorem I.8.1]{Peller}). 
This result  is the real line analogue of Nehari's classical
theorem describing Hankel operators on $\ell^2$ 
(see \cite[Theorem II.2.2.4]{Ni}, \cite[Theorem I.1.1]{Peller}
or \cite[Theorem 1.3]{Po}). An equivalent formulation of the
above result
is that 
\begin{equation}\label{H1}
Hank_2(\mathbb{R}_+)\simeq H^1(\Rdb)^*, 
\end{equation}
where
$H^1(\Rdb)\subseteq L^1(\Rdb)$ is the Hardy space
of all integrable functions whose Fourier transform vanishes on
$\Rdb_-$.

The first main result of this paper
is that for any $1<p<\infty$,
the Banach space $Hank_p(\mathbb{R}_+)$ coincides with
$\overline{{\rm Span}}^{w^*}\{\theta_u\, :\,
u>0\}\subset B(L^p(\Rdb_+))$
where, for any $u>0$, $\theta_u\colon L^p(\Rdb_+)\to L^p(\Rdb_+)$
is the Hankel operator
defined by $\theta_u f = f(u-\,\cdotp)$. 
As a consequence, we show that 
\begin{equation}\label{Dual}
Hank_p(\mathbb{R}_+)\simeq A_p(\Rdb_+)^*, 
\end{equation}
where
$A_p(\Rdb_+)$ is a half-line analogue of the 
Figa-Talamenca-Herz algebra
$A_p(\Rdb)$ (see e.g. \cite[Chapter 3]{Der}). 
We will see in Remark \ref{ReHank} (a) that 
$A_2(\Rdb_+)\simeq H^1(\Rdb)$. Thus, the
duality result (\ref{Dual}), established in 
Theorem \ref{dual space}, is an $L^p$-version of (\ref{H1}).

By a multiplier of $Hank_p(\mathbb{R}_+)$, we mean 
a $w^*$-continuous operator 
$T\colon Hank_p(\mathbb{R}_+)\to Hank_p(\mathbb{R}_+)$ 
such that $T(\theta_u)=m(u)\theta_u$ for all $u>0$,
for some function $m\colon\Rdb_+^*\to\Cdb$. In this case,
we set
$T=T_m$  and it turns out
that $m$ is necessarily bounded and continuous, see Lemma 
\ref{Continuous}.
The second main result of this paper 
is a characterization of 
of $p$-completely bounded multipliers $T_m$. We refer to 
Section 2 for some background on $p$-complete boundedness, 
whose definition goes back to \cite{Pisier90} (see
also \cite{Daws,Le Merdy,Pis}). We prove in Theorem \ref{last} that
$T_m\colon Hank_p(\mathbb{R}_+)\to Hank_p(\mathbb{R}_+)$ 
is a $p$-completely bounded multiplier
if and only if there exist a measure space
$(\Omega,\mu)$ and two essentially bounded measurable
functions $\alpha\colon\Rdb_+\to L^p(\Omega)$ and
$\beta\colon\Rdb_+\to L^{p'}(\Omega)$ such that
$m(s+t)=\langle\alpha(s),\beta(t)\rangle$ for almost
every $(s,t)\in\Rdb_+^{*2}$. This is a generalisation
of \cite[Theorem 3.1]{ALZ}. Indeed, the result in \cite{ALZ}
provides a characterization of $S^1$-bounded
multipliers on $H^1(\Rdb)$. Using (\ref{H1}), this yields
a characterization of completely bounded multipliers
on $Hank_2(\mathbb{R}_+)$, which is nothing but the case $p=2$ of 
Theorem \ref{last}. See Remark
\ref{p=2} for more on this.

Let us briefly explain the plan
of the paper. Section 2 contains some preliminary results. 
Section 3 is devoted to 
$Hank_p(\mathbb{N})\subset B(\ell^p)$, the space of Hankel
operators on $\ell^p$. 
We establish analogues of the aforementioned
results in the discrete setting. Results for $Hank_p(\mathbb{N})$
are easier than those concerning $Hank_p(\mathbb{R}_+)$
and Section 3 can be considered as a warm up. The main
results are stated and proved in Section 4.

\section{Preliminaries} \label{sec:preleminaries}
All our Banach spaces are complex ones. For any Banach spaces
$X,Z$, we let $B(X,Z)$ denote the Banach space of all bounded operators 
from $X$ into $Z$ and we write $B(X)$ instead of $B(X,X)$ when $Z=X$.
For any $x\in X$ and $x^*\in X^*$, the duality
action $x^*(x)$ is denoted by
$\langle x^*,x\rangle_{X^*,X}$, or simply by
$\langle x^*,x\rangle$ if there is no risk of confusion.

We start with duality on tensor products.
Let $X,Y$ be Banach spaces. Let 
$X \widehat{\otimes} Y$ denote their projective tensor product
\cite[Section VIII.1]{DU}. We will use the 
classical isometric identification
\begin{equation}\label{Duality} 
(X \widehat{\otimes} Y)^*\simeq B(X,Y^*) 
\end{equation}
provided e.g. by \cite[Corollary VIII.2.2]{DU}. More precisely, 
for any $\xi\in (X \widehat{\otimes} Y)^*$, there exists a
necessarily unique $R_\xi\in B(X,Y^*)$ such that 
$\xi(x\otimes y)=\langle R_\xi(x),y\rangle$ for all
$x\in X$ and $y\in Y$. Moreover $\norm{R_\xi}=\norm{\xi}$ and
the mapping $\xi\mapsto R_\xi$ is onto.

\begin{lem}\label{approx}
Let $A\subset X$ and $B\subset Y$ such that 
${\rm Span}\{A\}$ is dense in $X$ and 
${\rm Span}\{B\}$ is dense in $Y$. Assume that
$(R_\iota)_\iota$ is a bounded net of 
$B(X,Y^*)$. Then $R_\iota$ converges to some
$R\in B(X,Y^*)$ in the $w^*$-topology if and
only if $\langle R_\iota(x),y\rangle
\to \langle R(x),y\rangle$, for all 
$x\in A$ and $y\in B$.
\end{lem}

\begin{proof}
Assume the latter property. Since
the algebraic tensor
product $X\otimes Y$ is dense in $X \widehat{\otimes} Y$,
it implies that $\langle R_\iota,z\rangle
\to \langle R ,z\rangle$, for all $z$ belonging
to a dense subspace of $X \widehat{\otimes} Y$. Next, 
the boundedness of $(R_\iota)_\iota$ implies that 
$\langle R_\iota,z\rangle
\to \langle R ,z\rangle$, for all $z$ belonging to 
$X \widehat{\otimes} Y$. The equivalence follows.
\end{proof}

We will use the above duality principles in the case when 
$X=Y^*$ is an $L^p$-space $L^p(\Omega)$, for some 
index  $1<p<\infty$.

We now give a brief background on $p$-completely bounded maps, following \cite{Pisier90} (see also
\cite{Daws, Le Merdy, Pis}). Let $1<p<\infty$ and let 
$SQ_p$ denote the collection of quotients of subspaces of $L^p$-spaces, where we identify spaces which are isometrically isomorphic. Let $E$ be an $SQ_p$-space. Let 
$n\geq 1$ be an integer and let $[T_{ij}]_{1\leq i,j\leq n}
\in M_n\otimes B(E)$ be an $n\times n$ matrix with entries
$T_{ij}$ in $B(E)$. We equip $M_n\otimes B(E)$ with the
norm defined by
\begin{equation}\label{pNorm}
\bignorm{[T_{ij}]} = \sup\Bigl\{\Bigl(\sum_{i=1}^n \Bignorm{\sum_{j=1}^{n}
T_{ij}(x_j)}^p\Bigr)^{\frac{1}{p}}\, :\, x_1,\ldots,x_n\in E,\ \sum_{i=1}^n \norm{x_i}^p\leq 1\Bigr\}.
\end{equation}
If $S\subset B(E)$ is any subspace, then we let
$M_n(S)$ denote 
$M_n\otimes S$ equipped with the induced norm.

Let $S_1$ and $S_2$ be subspaces of $B(E_1)$ and $B(E_2)$, respectively, for
some $SQ_p$-spaces $E_1$ and $E_2$. 
Let $w\colon S_1\to S_2$ be a linear map. 
For any $n\geq 1$, let $w_n\colon M_n(S_1)\to M_n(S_2)$ be defined
by $w_n\bigl([T_{ij}]\bigr)= [w(T_{ij})]$,
for any $[T_{ij}]_{1\leq i,j\leq n}\in M_n(S_1)$.
By definition, $w$
is called $p$-completely bounded if 
the maps $w_n$ are uniformly bounded. In this case, 
the $p$-$cb$  norm of $w$ is defined by
$\|w\|_{p-cb}=\sup_n\|w_n\|$. We further say that $w$ is $p$-completely contractive if $\|w\|_{p-cb}\leq 1$
and that $w$ is a $p$-complete isometry if $w_n$ is an isometry
for all $n\geq 1$.
Note that the case $p=2$ corresponds to the classical notion of completely bounded maps (see e.g. \cite{Paulsen, Pis}).

We recall the following factorisation theorem of Pisier 
(see \cite[Theorem 1.4]{Le Merdy} and \cite[Theorem 2.1]{Pisier90}),
which extends Wittstock's factorization theorem \cite[Theorem 8.4]{Paulsen}.

\begin{thm}\label{PW}
Let $(\Omega_1,\mu_1)$ and 
$(\Omega_2,\mu_2)$ be measure spaces and 
let $1<p<\infty$. Let $S\subseteq B(L^p(\Omega_1))$
be a unital subalgebra. Let $w\colon S\to B(L^p(\Omega_2))$ be a 
linear map and let $C\geq 0$ be a constant.
The following assertions are equivalent.
\begin{itemize}
\item [(i)] The map $w$ is
$p$-completely bounded and $\|w\|_{p-cb}\leq C$. 
\item [(ii)] 
There exist an $SQ_p$-space $E$, a unital, non degenerate
$p$-completely contractive homomorphism $\pi:S\to B(E)$ 
as well as operators $V\colon L^p(\Omega_2)\to E$ and 
$W\colon E\to L^{p}(\Omega_2)$ such that $\|V\|\|W\|\leq C$ 
and for any $x\in S$, $w(x)=W\pi(x) V$.
\end{itemize}
\end{thm}

\begin{rem}\label{Duality-SQp}
Let $1<p<\infty$ and let $p'$ be its conjugate index. 
Let $E$ be an $SQ_p$-space. Then by assumption, there
exist a measure space $(\Omega,\mu)$ and two closed subspaces
$E_2\subseteq E_1\subseteq L^p(\Omega)$ such 
that $E=\frac{E_1}{E_2}$. Then 
$E_1^\perp\subseteq E_2^\perp\subseteq 
L^{p'}(\Omega)$ and we have an isometric 
identification
\begin{equation}\label{Dual-SQp}
E^*\simeq \frac{E_2^\perp}{E_1^\perp},
\end{equation}
by the classical duality between subspaces and 
quotients of Banach spaces.
More explicitly, let $f\in E_1$ and let $g\in E_2^\perp$. Let 
$\dot{f}\in E$ denote the class of $f$ modulo $E_2$
and let $\dot{g}\in E^*$ denote the element associated
to the class of $g$ modulo $E_1^\perp$ through the 
identification (\ref{Dual-SQp}).
Then we have
\begin{equation}\label{Lifting}
\langle\dot{g},\dot{f}\rangle_{E^*,E}\, =\, \langle g,f\rangle_{L^{p'},L^p}.
\end{equation}
\end{rem}

We now turn to Bochner spaces.
Let $(\Sigma,\nu)$ be a measure space and
let $X$ be a Banach space. For any $1\leq p\leq\infty$,
we let $L^p(\Sigma;X)$ denote the space of 
all measurable functions $\phi\colon\Sigma\to X$
(defined up to almost everywhere zero functions)
such that the norm function $t\mapsto \norm{\phi(t)}$
 belongs to $L^p(\Sigma)$. This 
 is a Banach space for the norm $\norm{\phi}_p$,
 defined as the 
 $L^p(\Sigma)$-norm of $\norm{\phi(\cdotp)}$ 
 (see e.g. \cite[Chapters I and II]{DU}).

 Assume that $p$ is finite and note that in this
 case, $L^p(\Sigma)\otimes X$ is dense
 in $L^p(\Sigma;X)$. Let $p'$ be the conjugate 
index of $p$. 
For all $\phi\in L^p(\Sigma;X)$ and $\psi\in L^{p'}(\Sigma;X^*)$,
the function $t\mapsto\langle\psi(t),\phi(t)\rangle_{X^*,X}$ belongs to
$L^1(\Sigma)$ and the resulting duality paring 
$\langle\psi,\phi\rangle := \int_\Omega\langle\psi(t),\phi(t)\rangle\, d\nu(t)$ extends to an isometric embedding 
$L^{p'}(\Sigma;X^*)\hookrightarrow L^p(\Sigma;X)^*$. Furthermore,
this embedding is onto if $X$ is reflexive, that is,
\begin{equation}\label{DualBochner}
L^{p'}(\Sigma;X^*)\simeq L^p(\Sigma;X)^*
\qquad\hbox{if}\ X\ \hbox{is reflexive}.
\end{equation}
We refer to \cite[Corollary III.2.13 $\&$ Section IV.1]{DU} for these results and complements.

Let $(\Sigma,\nu)$ and $(\Omega,\mu)$ be two measure spaces.
Then we have an isometric identification
$$
L^p(\Sigma;L^p(\Omega)) \simeq L^p(\Sigma\times\Omega),
$$
from which it follows that for any $T\in B(L^p(\Sigma))$,
the tensor extension $T\otimes I_{L^p(\Omega)}\colon
L^p(\Sigma)\otimes L^p(\Omega)\to L^p(\Sigma)
\otimes L^p(\Omega)$ 
extends to a bounded 
operator $T\overline{\otimes} I_{L^p(\Omega)}$ on
$L^p(\Sigma\times\Omega)$, whose norm is equal to the norm of $T$.
The following is elementary.

\begin{lem}\label{Tensor-Ext} 
The mapping $\pi\colon B(L^p(\Sigma))\to B(L^p(\Sigma\times\Omega))$
defined by $\pi(T)= T\overline{\otimes} I_{L^p(\Omega)}$
is a $p$-complete isometry.
\end{lem}

\begin{proof}
Let $n\geq 1$ and let $J_n=\{1,\ldots,n\}$. It follows from (\ref{pNorm}) that $M_n(B(L^p(\Sigma)))=B(\ell^p_n(L^p(\Sigma)))$
and hence $M_n(B(L^p(\Sigma)))
=B(L^p(J_n\times\Sigma))$ isometrically. Likewise, we have
$M_n(B(L^p(\Sigma\times\Omega)))
=B(L^p(J_n\times\Sigma\times\Omega))$ isometrically. 
Through these
identifications, 
$$
\Bigl[T_{ij}\overline{\otimes} 
I_{L^p(\Omega)}\Bigr] = 
[T_{ij}]\overline{\otimes} I_{L^p(\Omega)},
$$
for all $[T_{ij}]_{1\leq i,j\leq n}$ in $M_n(B(L^p(\Sigma)))$.
The result follows at once.
\end{proof}

We finally state an important result concerning 
Schur products on $B(\ell^p_I)$-spaces.
Let $I$ be an index set, let $1<p<\infty$
and let $\ell^p_I$ denote the discrete $L^p$-space over $I$.
Let $(e_t)_{t\in I}$ be its canonical basis. 
To any $T\in B(\ell^p_I)$, we associate 
a matrix of complex numbers, $[a_{st}]_{s,t\in I}$, defined by
$a_{st}=\langle T(e_t),e_s\rangle$, for all $s,t\in I$.
Following \cite[Chapter 5]{Pis}, we say that 
a bounded family
$\{\varphi(s,t)\}_{(s,t)\in I^2}$ 
of complex numbers is a bounded Schur multiplier
on $B(\ell^p_I)$ if for all $T\in B(\ell^p_I)$, with
matrix $[a_{st}]_{s,t\in I}$, the matrix $[\varphi(s,t)
a_{st}]_{s,t\in I}$ represents an element of $B(\ell^p_I)$.
In this case, the mapping $[a_{st}]\to [\varphi(s,t)
a_{st}]$ is a bounded operator from $B(\ell^p_I)$ into
itself. We note that $\{\varphi(s,t)\}_{(s,t)\in I^2}$ 
is a bounded Schur multiplier with norm 
$\leq C$ if and only if
for all $n\geq 1$, for all $[a_{ij}]_{1\leq i,j\leq n}$
in $M_n$ and for all
$t_1,\ldots,t_n, s_1,\ldots,s_n$ in $I$,
we have
\begin{equation}\label{Schur}
\bignorm{\bigl[\varphi(s_i,t_j)a_{ij}
\bigr]}_{B(\ell^p_n)}
\leq C\bignorm{[a_{ij}]}_{B(\ell^p_n)}.
\end{equation}
In the sequel, we apply the above definitions 
to the case when $I=\Rdb_+^*$.

\begin{thm}\label{Herz}
Let $\varphi\colon \Rdb_+^{*2}\to\Cdb$ be a continuous
bounded function. Let $1<p,p'<\infty$ be conjugate indices and
let $C\geq 0$ be a constant. The following assertions are equivalent.
\begin{itemize}
\item [(i)] The family$\{\varphi(s,t)\}_{(s,t)\in {\mathbb R}_{+}^{*2}}$ 
is a bounded Schur multiplier
on $B(\ell^{p}_{{\mathbb R}_{+}^{*}})$, with norm $\leq C$. 
\item [(ii)] There exist a measure space $(\Omega,\mu)$
as well as two
functions $\alpha\in L^\infty(\Rdb_+;L^p(\Omega))$ and $\beta\in L^\infty(\Rdb_+;L^{p'}(\Omega))$ such that 
$\norm{\alpha}_\infty\norm{\beta}_\infty\leq C$ and
$\varphi(s,t)=\langle\alpha(s),\beta(t)\rangle_{L^p,L^{p'}}$ for almost every $(s,t)\in \Rdb_+^{*2}$.
\end{itemize}
\end{thm}

\begin{proof} According to \cite[Section 4.1]{Coine},
(ii) is equivalent to the fact that 
as an element of $L^\infty(\Rdb_+^{2})$,
\begin{itemize}
\item [(ii')] $\varphi$ is a 
 bounded Schur multiplier on $B(L^p(\Rdb_+))$.
\end{itemize}
It further follows from \cite[Lemma 1 and Lemma 2]{Herz} 
that since $\varphi$ is continuous, (ii') is equivalent
to (i). The result follows.
\end{proof}

%%%%%%%%%%%%%%%%%%%%%%%%%%%%%%%%%%%%%%%%%%%%%%%%%%%%%%%%%%%%%%%%%%%%%
\section{Hankel operators on $\ell^p$ and their multipliers}
In this section we work on the sequence spaces
$\ell^p=\ell^p_{\tiny\Ndb}$, where
$\Ndb=\{0,1,\ldots\}$. For any $1<p<\infty$,
we let $(e_n)_{\geq 0}$ denote the classical
basis of $\ell^p$. For any
$T\in B(\ell^p)$, the associated matrix 
$[t_{ij}]_{i,j\geq 0}$ is given by
$t_{ij}=\langle T(e_j),e_i\rangle$, for all $i,j\geq 0$.

Let $Hank_p(\mathbb{N})\subseteq B(\ell^p)$ be the subspace of 
all $T\in B(\ell^p)$ whose matrix is Hankelian, i.e. 
has the form $[c_{i+j}]_{i,j\geq 0}$ for some 
sequence $(c_k)_{k\geq 0}$ of complex numbers.

Let $p'$ be the conjugate index of 
$p$ and regard $\ell^{p}\otimes \ell^{p'}\subset B(\ell^p)$
in the usual way. We set
\begin{equation*}
\gamma_k=\sum_{i+j=k}e_i\otimes e_j
\end{equation*}
for any $k\geq 0$.
Then each $\gamma_k$ belongs to $Hank_p(\mathbb{N})$,
and $\norm{\gamma_k}=1$.
Indeed, the matrix of $\gamma_k$ is $[c_{i+j}]_{i,j\geq 0}$
with $c_k=1$ and $c_l=0$ for all $l\not=k$.

\begin{lem}\label{Hankp}
For any $1<p<\infty$,
the space $Hank_p(\mathbb{N})$ is the $w^*$-closure
of the linear span of $\{\gamma_k\,:\,k\geq0\}$.
\end{lem}

\begin{proof}
It is plain that $Hank_p(\mathbb{N})$ is a
$w^*$-closed subspace of $B(\ell^p)$, 
hence one inclusion is straightforward.

To check the other one, consider
$T\in Hank_p(\mathbb{N})$. By the definition of 
this space, there is a sequence $(c_k)_{k\geq 0}$ of $\mathbb{C}$ such that 
\begin{equation*}
\langle T(e_j),e_i \rangle=c_{i+j}, 
\quad\text{for all } i,j\geq0.
\end{equation*}
For any $n\geq1$, let $K_n$ be the Fejer kernel defined by
\begin{equation*}
K_n(t)=\sum_{k=-n}^{n}\left(1-\frac{\vert k\vert}{n}
\right) e^{int},\qquad t\in\Rdb.
\end{equation*}
Then let $T_n\in B(\ell^p)$ be the finite rank operator whose matrix is $\bigl[\widehat{K_n}(i+j)c_{i+j}\bigr]_{i,j\geq0}$. Note that 
$$
T_n=\sum_{k=0}^{n}\left(1-\frac{\vert k\vert}{n}\right)c_k\,
\gamma_k \,\in {\rm Span}\{\gamma_k\, :\, k\geq 0\}.
$$

We show that $\|T_n\|\leq\|T\|$. To see this, let 
$\alpha=(\alpha_j)_{j\geq0}\in\ell^p$ and 
$(\beta_m)_{m\geq0}\in\ell^{p'}$. We have that
\begin{align*}
\langle T_n(\alpha),\beta\rangle&=
\sum_{m,j\geq0}\widehat{K_n}(m+j)c_{m+j}\alpha_j\beta_m\\
&=\frac{1}{2\pi}\int_{-\pi}^{\pi}K_n(t)\sum_{m,j\geq0}c_{m+j}
\alpha_j\beta_m e^{-i(m+j)t}\, dt.
\end{align*}
Since $K_n\geq0$, we deduce
\begin{align*}
\bigl\vert \langle T_n(\alpha),
\beta\rangle\bigr\vert &\leq\frac{1}{2\pi}
\int_{-\pi}^{\pi}K_n(t)\,\biggl\vert
\sum_{m,j\geq0}c_{m+j}\alpha_j\beta_m e^{-i(m+j)t}\biggr\vert \,dt.
\end{align*}
Now  for all $t\in[-\pi,\pi]$, we have
\begin{align*}
\biggl\vert\sum_{m,j\geq0}c_{m+j}\alpha_j\beta_m 
e^{-i(m+j)t}\biggr\vert&=\Big\vert\sum_{m,j\geq0}c_{m+j}\left(e^{-ijt}
\alpha_j\right)\left(e^{-imt}\beta_m\right)\Big\vert\\
&=\Big\vert\Bigl\langle 
T\Bigl(\bigl(e^{-ijt}\alpha_j\bigr)_{j\geq0}\Bigr),
\bigl(e^{-imt}\beta_m\bigr)_{m\geq0}\Bigr\rangle\Big\vert\\
&\leq \|T\|\Bigl(\sum_{j\geq0}\vert e^{-ijt}\alpha_j\vert^p\Bigr)^{\frac{1}{p}}\Bigl(\sum_{m\geq0}\vert e^{-imt}\beta_m\vert^{p'}\Bigr)^{\frac{1}{p'}}\\
&\leq\|T\|\|\alpha\|_p\|\beta\|_{p'},
\end{align*}
Since $\frac{1}{2\pi}\int_{-\pi}^{\pi}K_n(t)dt=1$, 
we therefore obtain that $\vert\langle T_n(\alpha),
\beta\rangle\vert\leq\|T\|\|\alpha\|_p\|\beta\|_{p'}$. This proves  that $\|T_n\|\leq\| T\|$,
as requested.

For all $i,j\geq0$,
$$
\langle T_n(e_j),e_i\rangle = \widehat{K_n}(i+j)\,\langle Te_j,e_i\rangle
\longrightarrow \langle Te_j,e_i\rangle,
$$
when $n\to\infty$.
Hence $T_n\to T$ in the $w^*$-topology, by Lemma \ref{approx}.
Consequently, $T$ belongs to the $w^*$-closure of ${\rm Span}\{\gamma_k\, :\, k\geq 0\}.$
\end{proof}

\begin{rem}\label{DiscreteRem}

\ 

\smallskip
{\bf (a)\,} Nehari's celebrated theorem (see e.g.
\cite[Theorem II.2.2.4]{Ni}, \cite[Theorem I.1.1]{Peller}
or \cite[Theorem 1.3]{Po})) asserts that
\begin{equation}\label{Nehari}
Hank_2(\mathbb{N})\simeq \,\frac{L^\infty(\Tdb)}{H^\infty(\Tdb)}.
\end{equation}
Here $\Tdb$ stands for the unit circle of $\Cdb$ and
$H^\infty(\Tdb)\subset 
L^\infty(\Tdb)$ is the Hardy space of functions
whose negative Fourier coefficents vanish. The
isometric isomorphism $J\colon \frac{L^\infty(\small\Tdb)}{H^\infty(
\small\Tdb)}
\to Hank_2(\mathbb{N})$ providing (\ref{Nehari}) is defined as follows. Given any $F\in L^\infty(\Tdb)$,
let $\dot{F}$ denote its class modulo 
$H^\infty(\Tdb)$. Then $J(\dot{F})$ is 
the operator whose matrix is equal to 
$\bigl[\widehat{F}(-i-j-1)\bigr]_{i,j\geq 0}$.

{\bf (b)\,} We remark that $Hank_p(\mathbb{N})\subseteq Hank_2(\mathbb{N})$. To see this, note that if $T\in Hank_p(\mathbb{N})$, then because of the symmetry in its matrix representation due to being a Hankelian matrix, $T$ has the same matrix representation as $T^*$, and therefore $T$ extends to a bounded operator on $\ell^{p'}$. By interpolation, $T$ extends to a bounded operator on $\ell^2$, which is represented by the same matrix as $T$. Hence, $T$ belongs to $Hank_2(\mathbb{N})$.

However for $1<p\not=2<\infty$, there is no description
of $Hank_p(\mathbb{N})$ similar to Nehari's theorem.

\smallskip
{\bf (c)\,} The definition of $Hank_p(\mathbb{N})$ readily 
extends to the
case $p=1$. 
$$
Hank_1(\mathbb{N})\simeq\ell^1.
$$
isometrically.
Indeed, let $J_1\colon \ell^1\to Hank_1(\mathbb{N})$ be defined
by 
$$
J_1(c)=\sum_{k=0}^\infty c_k\gamma_k,\qquad 
c=(c_k)_{k\geq 0}\in \ell^1. 
$$
Next, let $J_2\colon
Hank_1(\mathbb{N})\to \ell^1$ be defined by $J_2(T)=T(e_0)$.
Then $J_1,J_2$ are contractions and it is easy
to check that they are inverse to each other.
Hence $J_1$ is an isometric isomorphism.
\end{rem}

We say that a sequence $m=(m_k)_{k\geq0}$ in $\mathbb{C}$ is 
the symbol of a multiplier
on $Hank_p(\mathbb{N})$ if there is a $w^*$-continuous operator $T_m:Hank_p(\mathbb{N})\to Hank_p(\mathbb{N})$ such that $$
T_m(\gamma_k)=m_k\gamma_k, 
\qquad k\geq 0.
$$
Note that such an operator is uniquely defined. In 
this case,  $m\in\ell^\infty$ and $\|m\|_\infty\leq\|T_m\|$.

The following is a
simple extension of \cite[Theorems 6.1 $\&$ 6.2]{Pis}.

\begin{thm}
Let $1<p<\infty$, let $C\geq 0$ be a constant and let
$m=(m_k)_{k\geq0}$ be a sequence in $\mathbb{C}$. 
The following assertions
are equivalent.
\begin{enumerate}
\item[(i)] $m$ is the symbol of a $p$-completely bounded multiplier 
on $Hank_p(\mathbb{N})$, and 
$$
\norm{T_m \colon
Hank_p(\mathbb{N})\longrightarrow 
Hank_p(\mathbb{N})}_{p-cb}\leq C.
$$
\item[(ii)] There exist a measure space $(\Omega,\mu)$, and bounded sequences $(\alpha_i)_{i\geq0}$ in $L^p(\Omega)$ 
and $(\beta_j)_{j\geq0}$ in $L^{p'}(\Omega)$ such that $m_{i+j}=\langle \alpha_i,\beta_j\rangle$, 
for every $i,j\geq0$, and 
$$
\sup_{i\geq 0}\norm{\alpha_i}_p \,
\sup_{j\geq 0}\norm{\beta_j}_{p'}\leq C. 
$$
\end{enumerate}
\end{thm}

\begin{proof} By homogeneity, we may assume that $C=1$ throughout this proof.

Assume $(i)$. Let $\kappa\colon
\ell^p_{\mathbb{Z}}\to\ell^p_{\mathbb{Z}}$ be defined by 
$\kappa((a_k)_{k\in\mathbb{Z}})=(a_{-k})_{k\in\mathbb{Z}}$, 
let $J\colon\ell^p_{\mathbb{N}}\to\ell^p_{\mathbb{Z}}$ 
be the canonical embedding and
let $Q\colon \ell^p_{\mathbb{Z}}\to\ell^p_{\mathbb{N}}$ 
be the canonical projection. 
Define $q\colon B(\ell^p_{\mathbb{Z}})\to B(\ell^p_{\mathbb{N}})$ 
by $q(T)=Q\kappa TJ$. 
According to the easy implication ``$(ii)\,
\Rightarrow\,(i)$" of Theorem \ref{PW}, the mapping 
$q$ is $p$-completely contractive. We note that if 
$[t_{i,j}]_{(i,j)\in\mathbb{Z}^2}$
is the matrix of some $T\in B(\ell^p_{\mathbb{Z}})$, 
then the matrix 
of $q(T)$ is equal to $[t_{-i,j}]_{(i,j)\in\mathbb{N}^2}$.

Let $\M_p(\mathbb{Z})\subseteq B(\ell^p_{\mathbb{Z}})$ be the space of 
all bounded Fourier multipliers on $\ell^p_{\mathbb{Z}}$; this
is a unital subalgebra.
Let $T\in \M_p(\mathbb{Z})$ 
and let $\phi\in L^\infty(\mathbb{T})$ denote its symbol. 
Then the  matrix of $T$ is equal
to $[\widehat{\phi}(i-j)]_{(i,j)\in\mathbb{Z}^2}$, hence 
the matrix of $q(T)$ is equal
to $[\widehat{\phi}(-i-j)]_{(i,j)\in\mathbb{N}^2}$. 
Hence, $q(T)$ is Hankelian. 
We can therefore consider the restriction map 
$$
q_{\vert{{\mathcal M}_p(\mathbb{Z})}}\colon \M_p(\mathbb{Z})\longrightarrow Hank_p(\mathbb{N}).
$$

Let $\mathtt{s}\colon\ell^p_{\mathbb{Z}}\to\ell^p_{\mathbb{Z}}$ be 
the shift operator defined by 
$\mathtt{s}(e_j)=e_{j+1}$, for all $j\in\mathbb{Z}$. 
We observe (left to the reader)
that 
\begin{equation}\label{qsn}
q(\mathtt{s}^{-k})=\gamma_k,\qquad k\in\mathbb{N}.
\end{equation}

We assume that 
$T_m:Hank_p(\mathbb{N})\to Hank_p(\mathbb{N})$ 
is $p$-completely contractive. 
Consider $w\colon \M_p(\mathbb{Z})\to Hank_p(\mathbb{N})
\subseteq B(\ell^p)$ 
defined by $w:=T_m\circ q_{\vert{{\mathcal M}_p(\mathbb{Z})}}$. Then $w$ is $p$-completely contractive. 
Applying Theorem \ref{PW} to $w$, we obtain an $SQ_p$-space $E$, 
a contractive homomorphism $\pi:\M_p(\mathbb{Z})\to B(E)$ and 
contractive maps $V:\ell^p_{\mathbb{N}}\to E$ and 
$W:E\to\ell^{p}_{\mathbb{N}}$ such that
\begin{equation}\label{eq1}
w(T)=W\pi(T)V,\qquad T\in \M_p(\mathbb{Z}).
\end{equation}
Let $i,j\geq0$. By (\ref{qsn}), we have
\begin{equation*}
w(\mathtt{s}^{-(i+j)})=T_m(q(\mathtt{s}^{-(i+j)}))=
T_m(\gamma_{i+j})=m_{i+j}\gamma_{i+j},
\end{equation*}
hence  $\bigl\langle w(\mathtt{s}^{-(i+j)})e_i,e_j
\bigr\rangle=m_{i+j}$. Consequently, from \eqref{eq1}, we obtain that
\begin{equation*}
m_{i+j}=\langle\pi(\mathtt{s}^{-(i+j)})V(e_i),W^*(e_j)\rangle_{E,E^*}.
\end{equation*}
The mapping $\pi$ is multiplicative, hence this implies that
$$
m_{i+j}=\langle\pi(\mathtt{s}^{-i})V(e_i),
\pi(\mathtt{s}^{-j})^*W^*(e_j)\rangle_{E,E^*}.
$$
Set $x_i:=\pi(\mathtt{s}^{-i})V(e_i)\in E$ and 
$y_j:=\pi(\mathtt{s}^{-j})^*W^*(e_j)\in E^*$. 
Then, for all $i,j\geq0$ we have 
$\norm{x_i}\leq 1$, $\norm{y_j}\leq 1$ and
$m_{i+j}=\langle x_i,y_j\rangle_{E,E^*}$.

Let us now apply Remark \ref{Duality-SQp}. 
As in the latter, consider a measure 
space $(\Omega,\mu)$ and
closed subspaces
$E_2\subset E_1\subset L^p(\Omega)$ such that $E=\frac{E_1}{E_2}$. Recall (\ref{Dual-SQp}). 
For any $i\geq 0$, pick $\alpha_i\in E_1$ such that $\norm{\alpha_i}_p=
\norm{x_i}$ and $\dot{\alpha_i}=x_i$. Likewise, for any 
$j\geq 0$,  pick $\beta_j\in E_2^\perp$ such that $\norm{\beta_j}_{p'}=
\norm{y_j}$ and $\dot{\beta_j}=y_j$. Then for all $i,j\geq 0$, 
we both 
have $\norm{\alpha_i}_p\leq 1$, 
$\norm{\beta_j}_{p'}\leq 1$ and 
$m_{i+j}=\langle \alpha_i,\beta_j\rangle_{L^p,L^{p'}}$.
This proves $(ii)$.

Conversely, assume $(ii)$. 
By \cite[Corollary 8.2]{Pis}, the 
family $\{m_{i+j}\}_{(i,j)\in{\mathbb N}^2}$ induces
a $p$-completely contractive Schur multiplier on $B(\ell^p)$. 
It is clear that the restriction of this Schur multiplier maps
$Hank_p(\mathbb{N})$ into itself. More precisely, 
it maps $\gamma_k$ to $m_k\gamma_k$ for all $k\geq 0$. 
Hence $m$ is the symbol
of a $p$-completely contractive  multiplier on $Hank_p(\mathbb{N})$.
\end{proof}

%%%%%%%%%%%%%%%%%%%%%%%%%%%%Section 4%%%%%%%%%%%%%%%%%%%%%%%%%%%%%

\section{Hankel operators on $L^p(\mathbb{R}_+)$}
Throughout we let $1<p<\infty$ and we let 
$p'$ denote its conjugate index.
For any $u>0$, we set $\tau_u f:=f(\,\cdotp -u)$, for all 
$f\in L^1(\mathbb{R})+L^\infty(\mathbb{R})$. Let 
$$
Hank_p(\mathbb{R}_+)\subseteq B(L^p(\mathbb{R}_+))
$$
be the space of Hankelian operators on $L^p(\mathbb{R}_+)$,
consisting of all bounded operators $T:L^p(\mathbb{R}_+)\to L^p(\mathbb{R}_+)$ 
such that
\begin{equation*}
\langle T\tau_u f,g\rangle=\langle Tf,\tau_u g\rangle,
\end{equation*}
for all $f\in L^p(\mathbb{R}_+)$, 
$g\in L^{p'}(\mathbb{R}_+)$ and $u>0$.

For any $u>0$, let $\theta_u:L^p(\mathbb{R_+})
\to L^p(\mathbb{R}_+)$ be 
defined by $\theta_u f =f(u-\,\cdotp)$. 
Note that $\theta_u$ is a Hankelian operator on $L^p(\mathbb{R}_+)$. 
Indeed, for all $f\in L^p(\mathbb{R}_+)$, 
$g\in L^{p'}(\mathbb{R}_+)$ and $v>0$, we have 
$$
\langle \theta_u\tau_v f,g\rangle= \int_{v}^{u} f(u-s)g(s-v)\, ds= 
\langle \theta_u f,\tau_v g\rangle 
$$
if $v<u$, and $\langle \theta_u\tau_v f,g\rangle=
\langle \theta_u f,\tau_v g\rangle=0$ 
if $v\geq u$. The operators $\theta_u$ are 
the continuous counterparts of
the operators $\gamma_k$ from Section 3. From this
point of view, part (1) of 
Theorem \ref{dual space} below is an
analogue of Lemma \ref{Hankp}.
However its proof is more delicate.

We introduce a new space 
$A_p(\mathbb{R}_+)\subseteq C_0(\mathbb{R}_+)$ by
\begin{equation*}
A_p(\mathbb{R}_+):=\bigg\{F=\sum_{n=1}^\infty f_n\ast g_n
\,:\, f_n\in L^p(\mathbb{R}_+),\,
g_n\in L^{p'}(\mathbb{R}_+)\text{ and }\sum_{n=1}^\infty\|f_n\|_p\|g_n\|_{p'}<\infty\bigg\},
\end{equation*}
and we equip it with the norm 
\begin{equation}\label{Norm}
\|F\|_{A_p}=\inf\Bigl\{\sum_{n=1}^\infty \|f_n\|_p\|g_n\|_{p'}\Bigr\},
\end{equation}
where the infimum runs over all possible representations 
of $F$ as above. 
The space $A_p(\mathbb{R}_+)$ is a half-line analogue 
of the classical Figa-Talamenca-Herz algebra
$A_p(\Rdb)$, see e.g. \cite{Der}. The classical arguments showing 
that the latter is a Banach space show as well that (\ref{Norm}) is 
a norm on $A_p(\mathbb{R}_+)$ and that $A_p(\mathbb{R}_+)$ is a Banach space.

It follows from the above definitions that there exists
a (necessarily unique) contractive map
\begin{equation*}
Q_p:L^p(\mathbb{R}_+)\widehat{\otimes}L^{p'}(\mathbb{R}_+)\longrightarrow A_p(\mathbb{R}_+)
\end{equation*}
such that $Q_p(f\otimes g)=f\ast g$,
for all
$f\in L^p(\mathbb{R}_+)$ and $g\in L^{p'}(\mathbb{R}_+)$. 
Moreover $Q_p$ is a quotient map. Hence the adjoint 
\begin{equation*} 
Q_p^*:A_p(\mathbb{R}_+)^*\longrightarrow B(L^p(\mathbb{R}_+))
\end{equation*}
of $Q_p$ is an isometry. This yields an isometric identification
$A_p(\mathbb{R}_+)^*\simeq \ker(Q_p)^\perp(={\rm ran}(Q_p^*))$.

We observe that
\begin{equation}\label{KerPerp}
\ker(Q_p)^\perp=\overline{\rm Span}^{w^*}\{\theta_u: u>0\}.
\end{equation}
To prove this, we note that
\begin{equation}\label{Convol}
\langle\theta_u,f\otimes g\rangle=\langle\theta_u(f),
g\rangle= (f\ast g) (u),
\end{equation}
for all $f\in L^p(\mathbb{R})$, $g\in L^{p'}(\mathbb{R}_+)$ and
$u>0$.  
Hence, 
$$
\Bigl\langle \theta_u,\sum_{n=1}^\infty
f_n\otimes g_n\Bigr\rangle = 
\Bigl(\sum_{n=1}^{\infty} f_n\ast g_n\Bigr)(u)
$$
for all sequences $(f_n)_n$
in $L^p(\mathbb{R}_+)$ and $(g_n)_n$
in $L^{p'}(\mathbb{R}_+)$ such that 
$\sum_{n=1}^\infty\|f_n\|_p\|g_n\|_{p'}<\infty$, and all $u>0$.
This implies that
$\text{Span}\{\theta_u: u>0\}_\perp=\ker(Q_p)$, 
and (\ref{KerPerp}) follows.

\begin{thm}\label{dual space} \ 

\begin{itemize}
\item [(1)] 
The space $Hank_p(\mathbb{R}_+)$ is equal to the
$w^*$-closure of the linear span of $\{\theta_u: u>0\}$.
\item [(2)] We have an isometric identification
$$
Hank_p(\mathbb{R}_+)\simeq A_p(\Rdb_+)^*.
$$
\end{itemize}
\end{thm}

\begin{proof}
Part (2) follows from part (1) and the discussion preceding the statement 
of Theorem \ref{dual space}. For any $f\in L^p(\Rdb_+)$, 
$g\in L^{p'}(\Rdb_+)$ and $u>0$, the functionals
$T\mapsto \langle T\tau_u f,g\rangle$ and $T\mapsto
\langle Tf,\tau_u g\rangle$ are $w^*$-continuous on $B(L^p(\Rdb_+))$. Consequently,
$Hank_p(\mathbb{R}_+)$ is $w^*$-closed. Hence 
$Hank_p(\mathbb{R}_+)$ contains the $w^*$-closure of 
${\rm Span}\{\theta_u: u>0\}$. To prove the reverse inclusion, it
suffices to show, by (\ref{KerPerp}), that
$$
Hank_p(\mathbb{R}_+)\subset \ker(Q_p)^\perp.
$$

We will use a double approximation process. 
First, let $k,l$ in $C_c(\Rdb)$, the space of continuous functions
with compact support.
%such that 
%$$
%\|k\|_p=1,\ \|l\|_{p'}=1
%\qquad\hbox{and}\qquad \int_{\mathbb{R}}k(-s)l(s)=1.
%$$
To any $T\in B(L^p(\Rdb_+))$, we associate $T_{k,l}\in B(L^p(\Rdb_+))$
defined by 
\begin{equation*}
\langle T_{k,l}(f),g\rangle=\int_{\mathbb{R}}
\langle T(\tau_u k\cdot f),\tau_{-u}l\cdot g\rangle \,
du,\qquad 
f\in L^p(\mathbb{R}_+),\, g\in L^{p'}(\mathbb{R}_+).
\end{equation*}
We note that 
\begin{align*}
\int_{\mathbb{R}}\Big\vert\langle T(\tau_u k\cdot f),
\tau_{-u}l\cdot g\rangle\Big\vert \,du 
&\leq\| T\|_p\left(\int_{\mathbb{R}}\|\tau_u kf\|_p^p 
\, du\right)^{\frac{1}{p}}
\left(\int_{\mathbb{R}}\|\tau_{-u}lg\|_{p'}^{p'}
\, du\right)^{\frac{1}{p'}}\\
& = \|T\|_p\|f\|_p\|g\|_{p'}\|k\|_{p}\|l\|_{p'}.
\end{align*}
Thus, $T_{k,l}$ is well-defined and 
$\norm{T_{k,l}}\leq\norm{T}\|k\|_{p}\|l\|_{p'}$. 
We are going to show that
\begin{equation}\label{Show}
T\in 
Hank_p(\mathbb{R}_+)\,\Longrightarrow\, 
T_{k,l}\in \ker(Q_p)^\perp.
\end{equation}

%
%
%\begin{align*}
%\langle T_{k,l}(\tau_sf),g\rangle&=\int_{\mathbb{R}}\langle T(\tau_u k\cdot \tau_s f),\tau_{-u}l\cdot g\rangle du\\
%&=\int_{\mathbb{R}}\langle T\left(\tau_s\left(\tau_{u-s}k\cdot f\right)\right),\tau_{-u} l\cdot g\rangle du\\
%&=\int_{\mathbb{R}}\langle T\left(\tau_{u-s} k\cdot f\right),\tau_{s}\left(\tau_{-u} l\cdot g\right)\rangle du\\
%&=\int_{\mathbb{R}}\langle T\left(\tau_{u-s} k\cdot f\right),\tau_{s-u} l\cdot \tau_s g\rangle du\\
%&=\int_{\mathbb{R}}\langle T(\tau_u k\cdot f),\tau_{-u}l\cdot \tau_s g\rangle du=\langle T_{k,l}(f),\tau_s g\rangle,
%\end{align*}
%hence $T_{k,l}$ is a Hankelian operator. To see that %$T_{k,l}$ belongs to $\overline{\text{span}}^{w^*}%\{\theta_u:u>0\}$, it what follows, we show that it belongs %to $\ker(Q)^{\perp}$. 
%

Let $\alpha\in C_c(\mathbb{R}_+)^+$ such that $\|\alpha\|_1=1$. Let $R_\alpha\in B(L^p(\mathbb{R}_+))$ be defined by 
$$
R_{\alpha}(f)=\alpha\ast f,\qquad f\in L^p(\Rdb_+).
$$
We show that $(TR_{\alpha})_{k,l}$ belongs to $\ker(Q_p)^\perp$
if $T\in Hank_p(\mathbb{R}_+)$, and we use these auxiliary operators to establish 
(\ref{Show}).

We fix some $T\in Hank_p(\mathbb{R}_+)$.
Let $z\in\ker(Q_p)$. Since $C_c(\mathbb{R}_+)$ is both dense
in $L^p(\Rdb_+)$ and $L^{p'}(\Rdb_+)$, it follows e.g. from
\cite[Chapter 3, Proposition 6]{Der} that there exist sequences $(f_n)_{n\geq1}$ and $(g_n)_{n\geq1}$ in $C_c(\mathbb{R}_+)$ such that $\sum_{n=1}^\infty\|f_n\|_p\|g_n\|_{p'}<\infty$ and $z=\sum_{n=1}^\infty f_n\otimes g_n$. Since $z\in\ker(Q_p)$, we have $\sum_{n=1}^\infty f_n\ast g_n=0$, pointwise.

We write $R_\alpha f = \int_{\mathbb{R}_+}
f(s) \tau_s\alpha\,ds\,$
as a Bochner integral,
for all $f\in C_c(\Rdb_+)$. 
A simple application of Fubini's theorem leads to
$$
k\ast l\cdot f_n\ast g_n =\int_{\mathbb{R}}\int_{\mathbb{R}_+}
(\tau_u k \cdot f_n)(s)\tau_s(\tau_{-u} l \cdot g_n)\, dsdu,
$$
for all $n\geq 1$. We deduce that 
%
%Since
%\begin{align*}
%\sum_{n=1}^\infty\langle TR_{\alpha} (f_n),g_n\rangle&=\sum_{n=1}^\infty\langle T(f_n\ast \alpha),g_n\rangle\\
%&=\sum_{n=1}^\infty\int_{\mathbb{R}_+} \langle T\left(f_n(s)\tau_s\alpha\right),g_n\rangle ds\\
%&=\sum_{n=1}^\infty\int_{\mathbb{R}_+}\langle T(\tau_s\alpha),f_n(s)g_n\rangle ds\\
%&=\sum_{n=1}^\infty\int_{\mathbb{R}_+}\langle T(\alpha),f_n(s)\tau_s g_n\rangle ds=\sum_{n=1}^%\infty\langle T(\alpha),f_n\ast g_n\rangle=0,
%\end{align*}
%$TR_{\alpha}$ belongs to $\ker(Q_p)^\perp$. Now note that
%
\begin{align*}
\sum_{n=1}^\infty\langle
\left(TR_{\alpha}\right)_{k,l}(f_n) , g_n\rangle
&=\sum_{n=1}^\infty\int_{\mathbb{R}}\langle 
TR_{\alpha}(\tau_uk\cdot f_n),\tau_{-u}l\cdot g_n\rangle\, du\\
&=\sum_{n=1}^\infty\int_{\mathbb{R}}
\langle T(\left(\tau_uk\cdot f_n)\ast\alpha\right),\tau_{-u}l
\cdot g_n\rangle\, du\\
&=\sum_{n=1}^\infty\int_{\mathbb{R}}\int_{\mathbb{R}_+}
\left(\tau_uk\cdot f_n\right)(s) 
\langle T(\tau_s\alpha), \tau_{-u}l\cdot g_n\rangle\, ds du\\
%&=\sum_{n=1}^\infty\int_{\mathbb{R}}\int_{\mathbb{R}_+}\langle T(\tau_s\alpha),\left(\tau_u k\cdot f_n\right)(s)\tau_{-u}l\cdot g_n\rangle \, ds du\\
&=\sum_{n=1}^\infty\int_{\mathbb{R}}\int_{\mathbb{R}_+}
\langle T(\alpha), \left(\tau_u k\cdot f_n\right)(s)
\tau_s\left(\tau_{-u}l 
\cdot g_n\right)\rangle\, ds du\\
&=\sum_{n=1}^\infty\langle T(\alpha),k\ast 
l\cdot f_n\ast g_n\rangle\\
& = \Bigl\langle T(\alpha), k\ast l\cdot \sum_{n=1}^{\infty
}f_n\ast g_n
\Bigr\rangle=0.
\end{align*}
This shows that $(TR_\alpha)_{k,l}$ belongs to $\ker(Q_p)^\perp$.

For $z,f_n,g_n$ as above, write
$$
\sum_{n=1}^\infty\langle T_{k,l}(f_n),g_n\rangle
=
\sum_{n=1}^\infty\langle T_{k,l}(f_n),g_n\rangle-\sum_{n=1}^\infty\langle (TR_{\alpha})_{k,l}(f_n),g_n\rangle. 
$$
Then we have
\begin{align*}
\Big\vert\sum_{n=1}^\infty\langle T_{k,l}(f_n),g_n\rangle\Big\vert
& \leq\sum_{n=1}^\infty \int_{\mathbb{R}}\Big\vert\langle T
\left(\tau_u k\cdot f_n-\left(\tau_u k\cdot f_n\right)\ast\alpha\right),
\tau_{-u} l\cdot g_n\rangle\Big\vert\, du\\
&\leq\sum_{n=1}^\infty\|T\|\left(\int_\mathbb{R}\|\tau_u k\cdot 
f_n-\left(\tau_u k\cdot f_n\right)\ast\alpha\|_p^p \,
du\right)^{\frac{1}{p}}
\left(\int_{\mathbb{R}}\|\tau_{-u}l\cdot 
g_n\|_{p'}^{p'}\, du\right)^{\frac{1}{p'}}\\
&\leq\|T\|\|l\|_{p'}\sum_{n=1}^\infty
\|g_n\|_{p'}\left(\int_{\mathbb{R}}\|\tau_u k\cdot
f_n-\left(\tau_u k\cdot f_n\right)\ast\alpha\|_p^p 
\, du\right)^{\frac{1}{p}}.
\end{align*}

Recall that by assumption, $\alpha\geq 0$
and $\int_{\mathbb{R}_+}\alpha(s)ds=1$. 
Then we deduce from above that
\begin{align*}
\Big\vert\sum_{n=1}^\infty\langle T_{k,l}(f_n),g_n\rangle\Big\vert
& \leq \|T\|\|l\|_{p'} \sum_{n=1}^\infty \|g_n\|_{p'}\left(\int_{\mathbb{R}}
\Big\|\int_{\mathbb{R}_+}\alpha(s)\left(\tau_u k\cdot f_n-\tau_s
\left(\tau_u k\cdot f_n\right)\right)ds\Big\|_p^p
\, du\right)^{\frac{1}{p}}\\
&\leq \|T\|\|l\|_{p'}\sum_{n=1}^\infty\|g_n\|_{p'} 
\left(\int_{\mathbb{R}}
\int_{\mathbb{R}_+}\alpha(s)\big\|\tau_u k\cdot f_n-\tau_s
\left(\tau_u k\cdot f_n\right)\big\|_p^p
\, ds du\right)^{\frac{1}{p}}.
\end{align*}
The integral in the right hand-side satisfies
\begin{align*}
\Bigl(&\int_{\mathbb{R}}\int_{\mathbb{R}_+} \alpha(s)
\big\|\tau_u k\cdot f_n-\tau_s\left(\tau_u k\cdot f_n
\right)\big\|_p^p\, dsdu\Bigr)^{\frac{1}{p}}\\
&\leq \Bigl(\int_{\mathbb{R}}
\int_{\mathbb{R}_+}\alpha(s)\big\|\tau_u k\cdot f_n-\tau_{s+u} k\cdot f_n\big\|_p^pdsdu\Bigr)^{\frac{1}{p}}\\
& \quad+\Bigl(\int_{\mathbb{R}}
\int_{\mathbb{R}_+}\alpha(s)\big\|\tau_{s+u} k\cdot f_n-\tau_{s} 
\left(\tau_u k\cdot f_n\right)\big\|_p^p
\, dsdu\Bigr)^{\frac{1}{p}}\\
&\leq \Bigl(\int_{\mathbb{R}}
\int_{\mathbb{R}_+}\alpha(s)\big\|
\tau_u\left(\left( k-\tau_{s} k\right)
\cdot f_n\right)\big\|_p^p
\, dsdu\Bigr)^{\frac{1}{p}}
+\Bigl(\int_{\mathbb{R}}
\int_{\mathbb{R}_+}\alpha(s)\big\|\tau_{s+u} k\cdot \left(f_n-\tau_{s} f_n\right)\big\|_p^pds
du\Bigr)^{\frac{1}{p}}\\
&\leq\sup_{s\in supp(\alpha)}
\Bigl(\int_{\mathbb{R}}\big\| \tau_u
\left(k-\tau_{s} k\right)\cdot f_n\big\|_p^p
\, du
\Bigr)^{\frac{1}{p}}+\sup_{s\in supp(\alpha)}
\Bigl(\int_{\mathbb{R}}\big\| \tau_{s+u} k\cdot \left(f_n-\tau_{s} 
f_n\right)\big\|_p^p\, du\Bigr)^{\frac{1}{p}}\\
&=\sup_{s\in supp(\alpha)}\big\|k-\tau_{s} k\big\|_p 
\big\|f_n\big\|_p+\sup_{s\in supp(\alpha)}\|k\|_p\|f_n-\tau_s f_n\|_p.
\end{align*}
Hence we obtain that
$$
\Big\vert \sum_{n=1}^\infty\langle T_{k,l}(f_n),  g_n\rangle\Big\vert
\leq \|T\|\|l\|_{p'}\sum_{n=1}^{\infty}\|g_n\|_{p'}
\left(\sup_{s\in supp(\alpha)}\big\|k-\tau_{s} k\big\|_p 
\big\|f_n\big\|_p+\sup_{s\in supp(\alpha)}\|k\|_p\|f_n-\tau_s f_n\|_p\right).
$$

Given $\epsilon>0$, choose $M$ such 
that 
$$
\sum_{n=M+1}^\infty\|f_n\|_p\|g_n\|_{p'}< \epsilon.
$$
We may find
$s_0>0$ such that for all $s\in(0,s_0)$ and for all $1\leq n\leq M$, we have that 
\begin{equation*}
\|k-\tau_s k\|_p\leq\frac{\epsilon \|k\|_p}{\sum_{n=1}^\infty\|f_n\|_p\|g_n\|_{p'}}\quad\text{ and }\quad\|f_n-\tau_s f_n\|_p\leq\frac{\epsilon}{M\|g_n\|_{p'}}.
\end{equation*}
We may now choose $\alpha$ so that $supp(\alpha)\subseteq(0,t_0)$. Then we obtain from above that 
\begin{align*}
\Big\vert \sum_{n=1}^\infty\langle T_{k,l}(f_n),g_n\rangle\Big\vert
& \leq \|T\|\|l\|_{p'} \Bigl(\epsilon\|k\|_p
+\sum_{n=1}^{M} \|g_n\|_{p'}\cdot \sup_{s\in supp(\alpha)}\|k\|_p\|f_n-\tau_s f_n\|_p\\
& +\sum_{n=M+1}^{\infty} \|g_n\|_{p'}\cdot 
\sup_{s\in supp(\alpha)}\|k\|_p\|f_n-\tau_s f_n\|_p\Bigr)\\
&\leq \|T\|\|l\|_{p'} \Bigl(2 \epsilon\|k\|_p + \sum_{n=M+1}^{\infty} 2
\|k\|_p\|g_n\|_{p'}
\|f_n\|_p\Bigr)\\
& \leq 4\epsilon \|T\|\|l\|_{p'} \|k\|_p.
\end{align*}
Since $\epsilon$ was arbitrary,
this shows that $\sum_{n=1}^\infty\langle T_{k,l}(f_n),g_n\rangle=0$. Since 
$z=\sum_{n=1}^\infty f_n\otimes g_n$ was an arbitrary element of $\ker(Q_p)$, 
we obtain (\ref{Show}).

Next, we construct a sequence $(T_{k_n,l_n})_n$ 
which tends to $T$ in the $w^*$-topology of $B(L^p(\mathbb{R}_+))$. 
In the sequel, we assume that $k,l$ in $C_c(\Rdb)$
are such that 
\begin{equation}\label{Normal}
\|k\|_p=1,\ \|l\|_{p'}=1
\qquad\hbox{and}\qquad \int_{\mathbb{R}}k(-s)l(s)\, ds=1.
\end{equation}
Consider any $f,g\in C_c(\Rdb_+)$. We have
\begin{align*}
\Big\vert\langle T(f),g\rangle-\langle T_{k,l}(f),g\rangle\Big\vert
&=\Big\vert\int_{\mathbb{R}}\langle T\left(k(-s)f\right),
l(s)g\rangle -\langle T(\tau_s k\cdot f),\tau_{-s}l\cdot g\rangle 
\, ds\Big\vert\\
&\leq\int_{\mathbb{R}}\Big\vert\langle 
T\left(\left(k(-s)-\tau_s k\right)  
f\right),l(s)g\rangle
\Big\vert \, ds \\
&\quad +
\int_{\mathbb{R}} \Big\vert 
\langle T\left(\tau_s k \cdot f\right),
\left(l(s)-\tau_{-s}l\right)g\rangle\Big\vert\, ds\\
&\leq\|T\|\left(\int_{\mathbb{R}}\|\left(k(-s)-\tau_sk\right)f\|_p^p
ds\right)^{\frac{1}{p}}\left(\int_{\mathbb{R}}\|l(s)g\|_{p'}^{p'}
\, ds\right)^{\frac{1}{p'}}\\
&+\|T\|\left(\int_{\mathbb{R}}\|\tau_sk\cdot 
f\|_p^pds\right)^{\frac{1}{p}}
\left(\int_{\mathbb{R}}\|\left(l(s)-\tau_{-s}l\right)
g\|_{p'}^{p'}
\, ds\right)^{\frac{1}{p'}}\\
&\leq\|T\| \|g\|_{p'}\left(\int_{\mathbb{R}_+}\big\vert f(t)
\big\vert^p\|\tau_t\check{k}-\check{k}\|_p^p
\, dt\right)^{\frac{1}{p}}\\ 
& +\|T\|\|f\|_p\left(\int_{\mathbb{R}_+}\big\vert g(t)
\big\vert^{p'}\|\tau_{-t}l-l\|_{p'}^{p'}\, dt
\right)^{\frac{1}{p'}}.
\end{align*}
Here $\check{k}$ denotes the function $s\mapsto k(-s)$.

Now for $n\geq 1$, set $k_n:=\frac{\chi_{[-n,n]}}{(2n)^{\frac{1}{p}}}$ and 
$l_n:=\frac{\chi_{[-n,n]}}{(2n)^{\frac{1}{p'}}}$,
where $\chi_{[-n,n]}$ is the indicator function of the interval 
$[-n,n]$.
Then 
$\|k_n\|_p=\|l_n\|_{p'}=1$
and $\int_\mathbb{R}k_n(-s)l_n(s)\, ds=1$
as in (\ref{Normal}).
Let $K=supp(f)\cup supp(g)$ and let $r=\sup(K)$. Note that 
$\check{k_n}=k_n$ and that 
we have
\begin{align*}
\sup_{t\in K}\|\tau_t {k_n}-{k_n}\|_p\leq
\left(\frac{r}{n}\right)^{\frac{1}{p}}\quad\text{ and }
\quad\sup_{t\in K}\|\tau_{-t}l_n-l_n\|_{p'}\leq\left(\frac{r}{n}\right)^{\frac{1}{p'}}.
\end{align*}
Therefore,
\begin{align*}
\Big\vert\langle T(f),g\rangle-\langle T_{k_n,l_n}(f),
g\rangle\Big\vert\leq\frac{2r}{n}\|T\|\|f\|_p\|g\|_{p'},
\end{align*}
hence $\langle T_{k_n,l_n}(f),g\rangle\xrightarrow[n\to\infty]{}\langle T(f),g\rangle$.
Since $\|T_{k_n,l_n}\|\leq\|T\|$ for all $n\geq 1$, this implies, by Lemma \ref{approx}, that
$T_{k_n,l_n}\to T$ in the $w^*$-topology of $B(L^p(\mathbb{R}_+))$.
Consequently, $T\in\ker(Q_p)^\perp$ as expected.
\end{proof}

\begin{rem}\label{ReHank}

\ 

\smallskip
{\bf (a)\,} For any $1\leq p\leq \infty$, 
let $H^p(\Rdb)\subset L^p(\Rdb)$ be the subspace 
of all $f\in L^p(\Rdb)$ whose Fourier transform has support 
in $\Rdb_+$. Recall the factorisation property
\begin{equation*}
H^1(\mathbb{R})=H^2(\mathbb{R})\times H^2(\mathbb{R}).
\end{equation*}
More precisely, the product $h_1h_2\in H^1(\mathbb{R})$ 
and $\norm{h_1h_2}_1
\leq \norm{h_1}_2\norm{h_2}_2$ for all $h_1,h_2\in H^2(\mathbb{R})$
and conversely, for 
all $h\in H^1(\mathbb{R})$, there exist
$h_1,h_2\in H^2(\mathbb{R})$ such that $h=h_1h_2$ and 
$\norm{h}_1=\norm{h_1}_2\norm{h_2}_2$.

Recall that by definition,
\begin{equation*}
A_2(\mathbb{R}_+)=\Big\{\sum f_n\ast g_n: f_n,g_n\in L^2(\mathbb{R}_+), \sum\|f_n\|_2\|g_n\|_2<\infty\Big\}.
\end{equation*}
It therefore follows from the above factorisation property
and the identification of $L^2(\mathbb{R}_+)$ with $H^2(\mathbb{R})$
via the Fourier transform that 
\begin{equation*}
A_2(\mathbb{R}_+)=\big\{\hat{h}:h\in H^1(\mathbb{R})\},
\end{equation*}
with $\|\hat{h}\|_{A_2(\mathbb{R}_+)}=
\|h\|_{H^1(\mathbb{R})}$. Therefore, we have an isometric identification 
$$
A_2(\mathbb{R}_+)\cong H^1(\mathbb{R}).
$$
Since $H^1(\Rdb)^\perp = H^\infty(\Rdb)$, we have
$H_1(\mathbb{R})^*\cong
\frac{L^\infty(\mathbb{R})}{H^\infty(\mathbb{R})}$.
Applying Theorem \ref{dual space} (2), 
we recover the well-known fact  
(see \cite[Section IV.5.3]{Ni} or \cite[Theorem I.8.1]{Peller})
that 
$$
Hank_2(\mathbb{R}_+)\cong\frac{L^\infty(\mathbb{R})}{H^\infty(\mathbb{R})}.
$$

\smallskip
{\bf (b)\,} 
 We remark that $Hank_p(\mathbb{R}_+)\subseteq Hank_2(\mathbb{R}_+)$. 
Indeed, suppose that $T\in Hank_p(\mathbb{R}_+)$ and note that the adjoint mapping $T^*\in B(L^{p'}(\mathbb{R}_+))$ coincides with $T$ on $L^p(\mathbb{R}_+)\cap L^{p'}(\mathbb{R}_+)$. To see this, take $f,g \in L^{p}(\mathbb{R}_+)\cap L^{p'}(\mathbb{R}_+)$ and observe that
$f\otimes g -g\otimes f$ belongs to $\ker(Q_p)$. This implies
that $\langle T(f),g\rangle = \langle T(g),f\rangle$. 
Therefore, $T$ 
coincides with $T^*$ on $L^{p}(\mathbb{R}_+)\cap L^{p'}
(\mathbb{R}_+)$. It then follows by interpolation that 
$T$ extends to 
a bounded operator on $L^2(\mathbb{R}_+)$, 
say $\widetilde{T}$. Since 
$T$ and $\widetilde{T}$ coincide on $L^p(\mathbb{R}_+)\cap 
L^2(\mathbb{R}_+)$ and $T$ is Hankelian, it follows from the 
definition of Hankel operators that $\widetilde{T}$ is also a Hankel 
operator and hence belongs to $Hank_{2}(\mathbb{R}_+)$.

\smallskip
{\bf (c)\,} The definition of $Hank_p(\mathbb{R}_+)$ 
extends to the case $p=1$. In analogy with 
Remark \ref{DiscreteRem} (c), we have an isometric identification 
$$
Hank_1(\mathbb{R}_+)\simeq M(\Rdb_+^*),
$$
where $M(\Rdb_+^*)$ denotes the space of all
bounded Borel measures
on $\Rdb_+^*$. To establish this, we first note that 
for all $f\in L^1(\Rdb_+)$, the function
$u\mapsto \theta_u(f)$ is bounded and continuous
from $\Rdb_+^*$ into $L^1(\Rdb_+)$.
Hence for all
$\nu\in M(\Rdb_+^*)$,
we may define $H_\nu\in B(L^1(\Rdb_+))$ by
\begin{equation}\label{H}
H_\nu(f) = \int_{\mathbb{R}_+^*}\theta_u(f)\,d\nu(u),
\qquad f\in L^1(\Rdb_+).
\end{equation}
It is clear that $H_\nu$ is Hankelian. It follows from 
(\ref{Convol}) that 
$$
\bigl\langle H_\nu(f),g\bigr\rangle
= \int_{\mathbb{R}_+^*}(f\ast g)(u)\,\,d\nu(u),
\qquad f\in L^1(\Rdb_+),\, g\in L^\infty(\Rdb_+).
$$

We note that
the mapping $\nu\mapsto H_\nu$ is a 1-1 contraction
from $M(\Rdb_+^*)$ into $Hank_1(\mathbb{R}_+)$. We shall now
prove that this mapping is an onto isometry.

We use the isometric identification  
$M(\Rdb_+^*)\simeq C_0(\Rdb_+^*)^*$ provided by the Riesz theorem
and we regard $L^1(\Rdb_+)\subseteq M(\Rdb_+^*)$ 
in the obvious way. Let $T\in Hank_1(\mathbb{R}_+)$. 
We observe that for all $h,f\in L^1(\Rdb_+)$
and all $g\in C_0(\Rdb_+^*)$, we have
\begin{equation}\label{Commute}
\bigl\langle T(h\ast f),g\bigr\rangle
=\bigl\langle T(h),f\ast g\bigr\rangle
\end{equation}
Indeed, write $h\ast f=\int_{0}^\infty f(s)\tau_sh\, ds.$
This implies that $T(h\ast f)=\int_{0}^\infty
f(s) T(\tau_s h)\, ds$,
hence 
$$
\bigl\langle T(h\ast f),g \bigr\rangle =\int_{0}^\infty  f(s)
\langle T\tau_s h,g\rangle\, ds\,= \int_{0}^\infty f(s)
\langle Th,\tau_s g\rangle\, ds\,=
%\Bigl\langle Th,\int_{0}^\infty  f(s)\tau_sg \, ds\Bigr\rangle=
\bigl\langle T(h),f\ast g\bigr\rangle.
$$
Let $(h_n)_{n\geq 1}$ be a norm one
approximate unit of $L^1(\Rdb_+)$. Then $(T(h_n))_{n\geq 1}$
is a bounded sequence of $L^1(\Rdb_+)$. Hence it admits a 
cluster point $\nu\in M(\Rdb_+^*)$ in the $w^*$-topology of
$M(\Rdb_+^*)$. Thus, for all $g\in C_0(\Rdb_+^*)$, 
the complex number $\int_{{\mathbb R}_+^*} g(u)\, d\nu(u)\,$ is  
a cluster point of the sequence $(\langle T(h_n),g\rangle)_{n\geq 1}$.
Furthermore, we have $\norm{\nu}\leq\norm{T}$.
Let $f\in L^1(\Rdb_+)$ and let $g\in C_0(\Rdb_+^*)$.
Since $h_n\ast f\to f$ in $L^1(\Rdb_+),$ we have
that $\langle T(h_n\ast f),g\rangle\to 
\langle T (f),g\rangle$.
By (\ref{Commute}), we may write $\langle T(h_n\ast f),g\rangle=\langle T(h_n),f\ast g\rangle$. We deduce that
$$
\langle T (f),g\rangle=\int_{{\mathbb R}_+^*} 
(f\ast g)(u)\, d\nu(u).
$$
This implies that $T=H_\nu$, see (\ref{H}),
which concludes the proof.
\end{rem}

\begin{dfn}
We say that a function
$m:\mathbb{R}_+^*\to\mathbb{C}$ is the symbol of a 
multiplier on $Hank_p(\mathbb{R}_+)$ if there exist a $w^*$-continuous operator $T_m:Hank_p(\mathbb{R}_+)\to Hank_p(\mathbb{R}_+)$ 
such that for every $u>0$, $T_m(\theta_u)=m(u)\theta_u$.
(Note that such an operator  $T_m$ is necessarily unique.) 
\end{dfn}

\begin{rem}\label{pre-dual}
Suppose that $T_m
\colon Hank_p(\mathbb{R}_+)\to Hank_p(\mathbb{R}_+)$ is a multiplier
as defined above. Using Theorem \ref{dual space} (2),
let $S_m\colon
A_p(\mathbb{R}_+)\to A_p(\mathbb{R}_+)$ be the
operator such that $S_m^*=T_m$. 
For $f\in L^p(\mathbb{R}_+)$ and $g\in L^{p'}(\mathbb{R}_+)$, 
we have, by (\ref{Convol}),
\begin{align*}
[S_m(f\ast g)](u)&=\langle\theta_u,S_m(f\ast g)\rangle\\
&=\langle T_m(\theta_u),f\ast g\rangle\\
& =m(u)
\langle\theta_u,f\ast g\rangle\\
& =m(u)\left(f\ast g\right)(u).
\end{align*}
We deduce that $S_m(F)=m\cdot F$, 
for every $F\in A_p(\mathbb{R}_+)$.

Conversely, if $m\colon \mathbb{R}_+^*\to\mathbb{C}$ is such that $S_m\colon A_p(\mathbb{R}_+)\to A_p(\mathbb{R}_+)$ 
given by $S_m(F)=m\cdot F$ is well-defined and bounded, then 
$S_m^*$ is a multiplier on $Hank_p(\mathbb{R}_+)$.
\end{rem}

\begin{lem}\label{Continuous}
If $m\colon \mathbb{R}_+^*\to\mathbb{C}$ is the symbol of a
multiplier on $Hank_p(\mathbb{R}_+)$, then $m$ is continuous and bounded.
\end{lem}

\begin{proof}
For all $u>0$, we have $m(u)\theta_u=T_m(\theta_u)$, hence
$\vert m(u)\vert\leq\norm{T_m}$. Thus, $m$ is bounded. 
For any $a>0$, let $\chi_{(0,a)}$ be the indicator function
of the interval $(0,a)$. Then 
$m\cdot \chi_{(0,a)}*\chi_{(0,a)}$ belongs
to $A_p(\mathbb{R}_+)$, hence to $C_b(\Rdb_+^*)$, by Remark
\ref{pre-dual}. Since 
$\chi_{(0,a)}*\chi_{(0,a)}>0$ on $(0,2a)$, it follows that 
$m$ is continuous on $(0,2a)$. Thus, $m$ is continuous on
$\Rdb_+^*$.
\end{proof}

\begin{thm}\label{last}
Let $1<p<\infty$, let $C\geq 0$ be a 
constant and let
$m\colon \mathbb{R}_+^*\to\mathbb{C}$ be a function. 
The following assertions
are equivalent.
\begin{enumerate}
\item[(i)] $m$ is the symbol of a $p$-completely bounded
multiplier on $Hank_p(\mathbb{R}_{+})$, and 
$$
\norm{T_m \colon
Hank_p(\mathbb{R}_{+})\longrightarrow 
Hank_p(\mathbb{R}_{+})}_{p-cb}\leq C.
$$
\item[(ii)] $m$ is continuous and 
there exist a measure space $(\Omega,\mu)$ and two
functions $\alpha\in L^\infty(\mathbb{R}_{+}; L^p(\Omega))$
and $\beta\in L^\infty(\mathbb{R}_{+}; L^{p'}(\Omega))$ 
such that $\|\alpha\|_\infty\|\beta\|_\infty\leq C$ and 
$m(s+t)=\langle\alpha(s),\beta(t)\rangle$, 
for almost every $(s,t)\in\mathbb{R}_{+}^{*2}$.
\end{enumerate}
\end{thm}

\begin{proof}
By homogeneity, we may assume that 
$C=1$ throughout this proof.

Assume $(i)$. The continuity of $m$ follows from
Lemma \ref{Continuous}. Let $T_m:Hank_p(\mathbb{R}_+)\to Hank_p(\mathbb{R}_+)$ be 
the $p$-completely contractive multiplier associated 
with $m$. Let 
$\kappa\colon L^p(\mathbb{R})\to L^p(\mathbb{R})$ be defined by 
$(\kappa f)(t)=f(-t)$, for all $f\in L^p(\mathbb{R})$. 
Let $J\colon L^p(\mathbb{R}_+)\to L^p(\mathbb{R})$
be the canonical embedding
and let $Q\colon L^p(\mathbb{R})\to L^p(\mathbb{R}_+)$ be the canonical projection 
defined by $Qf=f_{\vert{\mathbb{R}_+}}$. 
Let $q\colon B(L^p(\mathbb{R}))\to B(L^p(\mathbb{R}_+))$ be given 
by $q(T)=Q\kappa TJ$, for all $T\in B(L^p(\mathbb{R}))$.
Applying the easy implication ``$(ii)\,\Rightarrow\,(i)$" 
of Theorem \ref{PW}
we obtain that $q$ is $p$-completely contractive.

Let $\M_p(\mathbb{R})\subseteq B(L^p(\mathbb{R}))$ 
denote the subalgebra of bounded Fourier multipliers. Let
us show that if  $T\in \M_p(\mathbb{R})$, 
then $q(T)\in Hank_p(\mathbb{R}_+)$. 
For any $s \in\mathbb{R}$, recall 
$\tau_s\in B(L^p(\mathbb{R}))$ given 
by $\tau_s(f)=f(\cdotp-s)$. 
Note that $\tau_s\in \M_p(\mathbb{R})$ 
and that $\M_p(\mathbb{R})=\overline{\rm Span}^{w^*}
\{\tau_s:s\in\mathbb{R}\}$.
For all $f\in L^p(\mathbb{R}_+)$,
we have 
\begin{align*}
q(\tau_s)f&=Q\tau(f(\cdotp-s))=Q(f(-(\cdotp +s)))=
\{t\in\mathbb{R}_+\mapsto f(-t-s)\}.
\end{align*}
Hence, if $s\geq0$, then $q(\tau_s)=0$ and if $s<0$, 
then $q(\tau_s)= \theta_{-s}$. It is plain
that $q$ is 
$w^*$-continuous. Since $Hank_p(\mathbb{R}_+)$
is $w^*$-closed, we deduce that
$q$ maps $\M_p(\mathbb{R})$ into
$Hank_p(\mathbb{R}_+)$.

Consider the mapping $q_0:=
q_{\vert{\mathcal M}_p(\mathbb{R})}
\colon \M_p(\mathbb{R})\to Hank_p(\mathbb{R}_+)$ and set
$\Gamma:= T_m \circ q_0\colon \M_p(\mathbb{R})\to B(L^p(\Rdb_+))$. 
It follows from above that 
\begin{equation}\label{Tm}
\Gamma(\tau_{-s})=m(s)\theta_s,\qquad s>0.
\end{equation}
Since $q$ is $p$-completely contractive,
$\Gamma$ is also  $p$-completely contractive. 
Applying Theorem \ref{PW} to $\Gamma$, we obtain
the existence of an $SQ_p$-space $E$, a 
unital $p$-completely contractive, non-degenerate homomorphism 
$\pi\colon \M_p(\mathbb{R})\to B(E)$ as well as operators 
$V\colon L^p(\mathbb{R}_+)\to E$ and $W\colon
E\to L^p(\mathbb{R}_+)$ 
such that $\|V\|\|W\|\leq1$ and for every 
$x\in \M_p(\mathbb{R})$, $\Gamma(x)=W\pi(x)V$.

Let $c\colon L^1(\mathbb{R})\to \M_p(\mathbb{R})$ 
be defined by $[c(g)](f)=g\ast f$, for all
$g\in L^1(\Rdb)$ and $f\in L^p(\Rdb)$. 
Let $\lambda:L^1(\mathbb{R})\to B(E)$ be given by 
$\lambda=\pi\circ c$. Then $\lambda$ is a 
contractive, non-degenerate homomorphism. 
By \cite[Remark 2.5]{DPR},
there exists $\sigma:\mathbb{R}\to B(E)$, a bounded strongly continuous representation such that for all 
$g\in L^1(\mathbb{R})$, $\lambda(g)=
\int_{\mathbb{R}} g(t)\sigma(t)\, dt\,$
(defined in the strong sense). 
Let us show that 
\begin{equation}\label{Gamma}
\Gamma(\tau_{-s})=W\sigma(-s)V,
\qquad s>0. 
\end{equation}

Let $\eta\in L^1(\mathbb{R})_+$ be such 
that $\int_{\mathbb{R}}\eta(t)\, dt=1$. For any $r>0$, 
let $\eta_r(t)=r\eta(rt)$. Since 
$\sigma:\mathbb{R}\to B(E)$ is strongly continuous, 
the function $t\mapsto \langle\sigma(t)x,x^*\rangle$
is continuous and we have 
\begin{equation}\label{eq2}
\int_{\mathbb{R}}\eta_r(-s-t)
\langle\sigma(t)x,x^*\rangle\, dt\,
\xrightarrow[r\to\infty]{}\langle\sigma(-s)x,x^*\rangle,
\end{equation}
for all $x\in E$ and $x^*\in E^*$. Since the left-hand side
in \eqref{eq2} is equal to 
$\langle\pi\left(c\left(\eta_r(-s-\,\cdotp)\right)\right)x,x^*\rangle$,
we obtain, by Lemma \ref{approx}, that 
$\pi\left(c\left(\eta_r(-s-\,\cdotp)\right)\right)
\to\sigma(-s)$ in the $w^*$-topology of $B(E)$.
This implies that $W\pi\left(c\left(\eta_r(-s-\,\cdotp)\right)
\right)V\to W\sigma(-s)V$  in the $w^*$-topology 
of $B(L^p(\mathbb{R}_+))$. 
We next show that $W\pi\left(c\left(\eta_r(-s-\,\cdotp)
\right)\right)V\to \Gamma (\tau_{-s})$ in 
the $w^*$-topology of $B(L^p(\mathbb{R}_+))$,
which will complete the proof of (\ref{Gamma}). 
Since 
$$
W\pi(c\left(\eta_r(-s-\,\cdotp)\right))V=
\Gamma\left(c(\eta_r(-s-\,\cdotp))\right)
$$
and $\Gamma$ is $w^*$-continuous, 
it suffices to show that $c\left(\eta_r(-s-\,\cdotp)\right)\to\tau_{-s}$ 
in the $w^*$-topology of $B(L^p(\mathbb{R}))$. 
To see this, let $f\in L^p(\mathbb{R})$ and 
$g\in L^{p'}(\mathbb{R})$. We have that
\begin{align*}
\langle c\left(\eta_r(-s-\,\cdotp)\right)f,g\rangle&=
\langle\eta_r(-s-\,\cdotp)\ast f,g\rangle\\
&=\langle\delta_{-s}\ast\eta_r\ast f,g\rangle\\
& \to
\langle\delta_{-s}\ast f,g\rangle=\langle\tau_{-s}f,g\rangle.
\end{align*}
By Lemma \ref{approx} again, this proves
that $c\left(\eta_r(-s-\,\cdotp)\right)\to\tau_{-s}$ in the
$w^*$-topology, as expected.

Given any $\epsilon>0$, let 
$m_\epsilon:\mathbb{R}_+^*\to\mathbb{C}$ be defined by 
$$
m_\epsilon(t)=m(t+\epsilon),\qquad t>0.
$$
Let $f\in L^p(\mathbb{R}_+)$ be given by  
$f=\epsilon^{-\frac{1}{p}}\chi_{(0,\epsilon)}$ and let $g\in L^{p'}
(\mathbb{R}_+)$ be given by 
$g=\epsilon^{-\frac{1}{p'}}\chi_{(0,\epsilon)}$. 
For any $s,t>0$, set 
$$
\alpha_\epsilon(s):=\sigma(-s-\tfrac{\epsilon}{2})
V(\tau_s f)\quad \text{and}\quad 
\beta_\epsilon(t):=\sigma(-t-\tfrac{\epsilon}{2})^*
W^*(\tau_t g).
$$
Since $\sigma$ is strongly continuous,
$\alpha_\epsilon$ and $\beta_\epsilon$
are continuous. 
By (\ref{Tm}) and (\ref{Gamma}), we have that
 \begin{align*}
\langle\alpha_\epsilon(s),\beta_\epsilon(t)\rangle_{E,E^*}
&=\langle\sigma(-s-\tfrac{\epsilon}{2})V(\tau_sf),
\sigma(-t-\tfrac{\epsilon}{2})^*W^*(\tau_tg)\rangle\\
&=\langle W\sigma(-s-t-\epsilon)V(\tau_s f),\tau_t g\rangle\\
&=\langle\left(\Gamma(\tau_{-s-t-\epsilon})\right)
(\tau_s f),\tau_t g\rangle\\
&=m(s+t+\epsilon)\langle\theta_{s+t+\epsilon}(\tau_sf),\tau_t g\rangle\\
&=m_\epsilon(s+t)\langle\epsilon^{-1/p}\chi_{(t,t+\epsilon)},\epsilon^{-1/p'}\chi_{(t,t+\epsilon)}\rangle\\ 
&=m_\epsilon(s+t),
\end{align*}
for all $s,t>0$.
Moreover, $\norm{\alpha_\epsilon(s)}\leq \norm{V}$ and
$\norm{\beta_\epsilon(t)}\leq \norm{W}$  for all $t,s>0$. Since
$\alpha_\epsilon$ and $\beta_\epsilon$ are continuous, 
this implies that
$\alpha_\epsilon\in L^\infty(\mathbb{R}_+;E)$, 
$\beta_\epsilon\in L^\infty(\mathbb{R}_+;E^*)$ and 
$\|\alpha_\epsilon\|_\infty\|\beta_\epsilon\|_\infty\leq\|V\|\|W\|\leq1$.

We now show that the $SQ_p$-space $E$ can be replaced by 
an $L^p$-space in the above factorization
property of $m_\epsilon$. Following Remark \ref{Duality-SQp}, 
assume that $E=\frac{E_1}{E_2}$, with 
$E_2\subseteq E_1\subseteq L^p(\Omega)$, and for all $f\in E_1$, let 
$\dot{f}\in E$ denote the class of $f$. Recall (\ref{Dual-SQp}) 
and for all $g\in E_2^\perp$,
let $\dot{g}\in E^*$ denote the class of $g$.
Since
$E$ is a quotient of $E_1$, we have an isometric
embedding $E^*\subseteq E_1^*$. More precisely,
$$
E^*=\frac{E_2^\perp}{E_1^\perp}\hookrightarrow
\frac{L^{p'}(\Omega)}{E_1^\perp}= E_1^*.
$$
This induces an isometric
embedding $L^1(\Rdb_+;E^*) \subseteq L^1(\Rdb_+;
E_1^*)$. Since $E^*$ and $E_1^*$ are reflexive,
we may apply the identifications
$L^1(\Rdb_+;E^*)^*\simeq L^\infty(\Rdb_+;E)$ and
$L^1(\Rdb_+;E_1^*)^*\simeq L^\infty(\Rdb_+;E_1)$
provided by (\ref{DualBochner}). By the Hahn-Banach theorem,
we deduce the existence of
$\widetilde{\alpha_\epsilon}\in L^\infty(\Rdb_+;E_1)$
such that $\norm{\widetilde{\alpha_\epsilon}}_\infty=
\norm{\alpha_\epsilon}_\infty$
and the functional $L^1(\Rdb_+;E_1^*)\to \Cdb$
induced by $\widetilde{\alpha_\epsilon}$ extends the functional
$L^1(\Rdb_+;E^*)\to \Cdb$
induced by $\alpha_\epsilon$. It is easy to check that the latter
means that $\dot{\widetilde{\alpha_\epsilon}(s)}
=\alpha_\epsilon(s)$ almost everywhere on $\Rdb_+$.
Likewise, there exist
$\widetilde{\beta_\epsilon}\in L^\infty(\Rdb_+;E_2^\perp)$
such that $\norm{\widetilde{\beta_\epsilon}}_\infty=
\norm{\beta_\epsilon}_\infty$
and $\dot{\widetilde{\beta_\epsilon}(t)}
=\beta_\epsilon(t)$ almost everywhere on $\Rdb_+$.
Regard $\widetilde{\alpha}_\epsilon$ as an element of
$L^\infty(\mathbb{R}_+,L^p(\Omega))$ and 
$\widetilde{\beta}_\epsilon$ as an 
element of $L^\infty(\mathbb{R}_+,L^{p'}(\Omega))$. 
By (\ref{Lifting}), we then have 
\begin{equation*}
\langle\alpha_{\epsilon}(s),\beta_{\epsilon}(t)\rangle_{E,E^*}
=\langle\widetilde{\alpha}_{\epsilon}(s),
\widetilde{\beta}_{\epsilon}(t)\rangle_{L^p,L^{p'}},
\end{equation*}
for almost every $(s,t)\in\Rdb_+^{*2}$.

We therefore obtain that $m_\epsilon:\mathbb{R}_+^*\to\mathbb{C}$ 
satisfies condition $(ii)$ of the theorem (with $C=1$).

Define $\varphi\colon\Rdb_+^{*2}\to\Cdb$ by $\varphi(s,t)=m(s+t)$. 
Likewise, for any $\epsilon>0$, define  
$\varphi_\epsilon\colon\Rdb_+^{*2}\to\Cdb$ by 
$\varphi(s,t)=m_\epsilon(s+t)$.
Since $m$ is continuous,
the functions $\varphi$ and $\varphi_\epsilon$ are continuous.
It follows from above that for all $\epsilon>0$, 
$\varphi_\epsilon$ satisfies condition $(ii)$ in 
Theorem \ref{Herz}, with $C=1$.
The latter theorem therefore implies that 
the family $\{\varphi_\epsilon(s,t)\}_{(s,t)\in{\mathbb R}_+^{*2}}$
is a bounded Schur multiplier on $B(\ell^p_{{\mathbb R}_+^{*}})$,
with norm less than one. Thus
for all $[a_{ij}]_{1\leq i,j\leq n}$
in $M_n$ and for all
$t_1,\ldots,t_n, s_1,\ldots,s_n$ in ${\mathbb R}_+^{*}$,
we have $\norm{[\varphi_\epsilon(s_i,t_j)a_{ij}]}_{B(\ell^p_n)}
\leq \norm{[a_{ij}]}_{B(\ell^p_n)}$.
Since $m$ is continuous, $\varphi_\epsilon\to \varphi$ pointwise
when $\epsilon\to 0$. We deduce that $\varphi$ satisfies
(\ref{Schur}) with $C=1$ for all $[a_{ij}]_{1\leq i,j\leq n}$
in $M_n$ and  all
$t_1,\ldots,t_n, s_1,\ldots,s_n$ in ${\mathbb R}_+^{*}$.
Consequently, the family $\{\varphi(s,t)\}_{(s,t)
\in{\mathbb R}_+^{*2}}$
is a bounded Schur multiplier on $B(\ell^p_{{\mathbb R}_+^{*}})$,
with norm less than one. Applying the 
implication ``$(i) \Rightarrow (ii)$" in Theorem \ref{Herz},
we deduce the assertion
$(ii)$ of Theorem \ref{last}.

 Conversely, assume $(ii)$. Following Lemma \ref{Tensor-Ext}, let 
 $\pi\colon B(L^p(\Rdb_+))\to B(L^p(\Rdb_+\times\Omega))$ be 
 the $p$-completely isometric homomorphism defined by
 $\pi(T)= T\overline{\otimes} I_{L^p(\Omega)}$. This map is
 $w^*$-continuous. Indeed, let $(T_\iota)_\iota$ be a bounded
 net of $B(L^p(\Rdb_+))$ converging to some
 $T\in B(L^p(\Rdb_+))$ in the $w^*$-topology. For any
 $f\in L^p(\Rdb_+)$, $g\in L^{p'}(\Rdb_+)$,
 $\varphi\in L^p(\Omega)$ and $\psi\in L^{p'}(\Omega)$,
 we have
 $$
 \bigl\langle \pi(T_\iota), (f\otimes\varphi)\otimes 
 (g\otimes\psi)\bigr\rangle = 
 \langle T_\iota f,g\rangle_{L^p({\mathbb R}_+), L^{p'}({\mathbb R}_+)}
 \langle \varphi, \psi\rangle_{L^p(\Omega), L^{p'}(\Omega)},
 $$
where the duality pairing in the left hand-side refers to the 
identification 
$$
\bigl(L^p(\Rdb_+\times\Omega)\widehat{\otimes}
L^{p'}(\Rdb_+\times\Omega)\bigr)^*
\simeq B((L^p(\Rdb_+\times\Omega)).
$$
Since $\langle T_\iota f,g\rangle\to \langle  Tf,g\rangle$,
we deduce that $\langle \pi(T_\iota), (f\otimes\varphi)\otimes 
(g\otimes\psi)\rangle 
\to \langle \pi(T), (f\otimes\varphi)\otimes 
 (g\otimes\psi)\rangle$. Since $L^p(\Rdb_+)\otimes L^p(\Omega)$
 and  $L^{p'}(\Rdb_+)\otimes L^{p'}(\Omega)$ are dense
 in $L^p(\Rdb_+\times\Omega)$ and $L^{p'}(\Rdb_+\times\Omega)$,
 respectively, we deduce that $\pi(T_\iota)\to \pi(T)$
 in the $w^*$-topology, by Lemma \ref{approx}.
 This proves that $\pi$ is
 $w^*$-continuous.

Let $V\colon L^p(\Rdb_+)\to L^p(\Rdb_+;L^p(\Omega))
\simeq L^p(\Rdb_+\times\Omega)$ be defined by 
$$
V(f)= f\alpha,\qquad f\in L^p(\Rdb_+).
$$
This is a well-defined contraction. Likewise we define
a contraction $W\colon L^{p}(\Rdb_+\times\Omega)\to L^{p}(\Rdb_+)$
by setting
$$
W^*(g)= g\beta,\qquad g\in L^{p'}(\Rdb_+).
$$
It follows from above and from the implication ``$(ii)\Rightarrow 
(i)$" of Theorem \ref{PW} that the mapping
$$
w\colon B(L^p(\Rdb_+))
\longrightarrow B(L^p(\Rdb_+)),
\qquad w(T) = W\pi(T)V,
$$
is a $w^*$-continuous $p$-complete contraction.

We claim that for all $u>0$, we have
\begin{equation}\label{wu}
w(\theta_u)=m(u)\theta_u.
\end{equation}
To prove this, consider $f\in L^p(\Rdb_+)$ and 
$g\in L^{p'}(\Rdb_+)$. For all $u>0$, we have
$$
\langle w(\theta_u)f,g\rangle =
\langle \pi(\theta_u)V(f),W^*(g)\rangle=
\langle \pi(\theta_u)(f\alpha),(g\beta)\rangle.
$$
By the definitions of $\pi$ and 
$\theta_u$, we have
$\pi(\theta_u)(f\alpha)=(f\alpha)(u-\,\cdotp)$. Consequently,
$$
\langle w(\theta_u)f,g\rangle =
\int_{0}^{u} f(u-t)g(t)\langle\alpha(u-t),\beta(t)\rangle\, dt,
\qquad u>0.
$$
Let $h\in L^1(\Rdb_+)$ be an auxiliary function. Then
using Fubini's theorem and setting $s=u-t$ in due place, 
we obtain that
\begin{align*}
\int_{0}^\infty 
\langle w(\theta_u)f,g\rangle
h(u)\, du\, 
&= \int_{0}^\infty 
\int_{t}^\infty h(u) 
f(u-t)g(t)\langle\alpha(u-t),\beta(t)\rangle\, dudt\\
&= \int_{0}^\infty 
\int_{0}^\infty h(s+t) 
f(s)g(t)\langle\alpha(s),\beta(t)\rangle\, dsdt.
\end{align*}
Applying the a.e. equality $m(s+t)=
\langle\alpha(s),\beta(t)\rangle$
and reversing this computation, we deduce that
$$
\int_{0}^\infty 
\langle w(\theta_u)f,g\rangle
h(u)\, du\, =\int_{0}^\infty 
m(u) (f\ast g)(u) 
h(u)\, du.
$$
Since $h$ is arbitrary, this implies that 
$\langle w(\theta_u)f,g\rangle=m(u)(f\ast g)(u)$ 
for a.e. $u>0$. Equivalently,
$\langle w(\theta_u)f,g\rangle=m(u)\langle \theta_u f,g\rangle$ 
for a.e. $u>0$.
It is plain that $u\mapsto \theta_u$
is $w^*$-continuous on $B(L^p(\Rdb_+))$.
Since $w$ is $w^*$-continuous, the function
$u\mapsto \langle w(\theta_u)f,g\rangle$
is continuous as well. Since $m$ is assumed continuous, 
we deduce that $\langle w(\theta_u)f,g\rangle=
m(u)\langle \theta_u f,g\rangle$
for all $u>0$. This yields (\ref{wu}),
for all $u>0$.

By part (1) of Theorem \ref{dual space} and the 
$w^*$-continuity of 
$w$, the identity (\ref{wu}) implies that $Hank_p(\Rdb_+)$ 
is an invariant subspace of $w$. Further the restriction of
$w$ to $Hank_p(\Rdb_+)$ is the multiplier associated to $m$. The
assertion $(i)$ follows.
%
%
%
 %Conversely, assume $(ii)$. By Remark \ref{pre-dual}, it is enough to show that for all $\phi$ in $A_p(\mathbb{R}_+)$, $m\cdot\phi$ also belongs to $A_p(\mathbb{R}_+)$ and that $\|m\cdot\phi\|_{A_p}\preceq \|\phi\|_{A_p}$. For this, it suffices to show that for all $f\in L^p(\mathbb{R}_+)$ and $g\in L^{p'}(\mathbb{R}_+)$, $m\cdot(f\ast g)\in A_p(\mathbb{R}_+)$ with $\|m\cdot(f\ast g)\|_{A_p}\leq\|\alpha\|_{\infty}\|\beta\|_\infty\|f\|_p\|g\|_{p'}$.
%
% Let $f\in L^p(\mathbb{R}_+)$ and $g\in L^{p'}(\mathbb{R}_+)$. Applying Lemma \ref{converse} to $F=f\alpha$ and $G=g\beta$, we get that 
%\begin{equation*}\left(f\alpha\odot g\beta\right)\inA_p(\mathbb{R}_+) \end{equation*}
%with $\|f\alpha\odot g\beta\|_{A_p}\leq\|f\alpha\|_p\|\beta\|_{p'}$. Now, note that
%\begin{align*}\Big[m\cdot (f\ast g)\Big](t)&=m(t)\int_0^t f(t)g(t-s)ds\\&=\int_0^t f(s)g(t-s)\langle\alpha(s),\beta(t-s)\rangle ds\\&\int_0^t\langle (f\alpha)(s),(g\beta)(t-s)\rangle ds=(F\odotG)(t).\end{align*}
 %Hence we obtain that $m\cdot (f\ast g)=F\odot G\in A_p(\mathbb{R}_+)$ and $$\|m\cdot(f\ast g)\|_{A_p}=\|F\odot G\|_{A_p}\leq\|f\alpha\|_p\|\beta g\|_{p'}\leq\|f\|_p\|\alpha\|_\infty\|g\|_{p'}\|\beta\|_\infty,$$ as desired. 
\end{proof}

\begin{rem}\label{p=2} 
We proved in \cite[Theorem 3.1]{ALZ}
that a continuous function $m\colon\Rdb_+^*\to\Cdb$ 
is the symbol of an 
$S^1$-bounded Fourier multiplier on $H^1(\Rdb)$, with
$S^1$-bounded norm $\leq C$, if and only if there 
exist a Hilbert space $\mathcal{H}$ and  two
functions $\alpha,\beta \in L^\infty(\mathbb{R}_+;\mathcal{H})$ 
such that $\|\alpha\|_\infty\|\beta\|_\infty\leq C$ and  
$m(s+t)=\langle\alpha(t),\beta(s)\rangle_{\mathcal H}$
for almost every $(s,t)\in\mathbb{R}_{+}^{*2}$. 
It turns out that using (\ref{H1}), 
a mapping $S\colon H^1(\Rdb)\to H^1(\Rdb)$ is an 
$S^1$-bounded Fourier multiplier with
$S^1$-bounded norm $\leq C$ if and only
if $S^*\colon Hank_2(\Rdb_+)\to Hank_2(\Rdb_+)$
is a completely bounded multiplier with
completely bounded norm $\leq C$. See
\cite[Remark 3.4]{ALZ} for more on this.
Thus the statement in \cite[Theorem 3.1]{ALZ}
is equivalent to the
case $p=2$ of Theorem \ref{last}. In this regard,
Theorem \ref{last} can be regarded
as a $p$-analogue of \cite[Theorem 3.1]{ALZ}.
\end{rem}

\begin{rem}
Let $f\in L^p(\Rdb_+)$ and $g\in L^{p'}(\Rdb_+)$. For any
$s,t>0$, we may write
$$
(f\ast g)(s+t)=\int_{\mathbb R} f(s+r)g(t-r)\, dr.
$$
Equivalently,
$$
(f\ast g)(s+t)=
\langle\tau_{-s}f,\tau_t\check{g}
\rangle_{L^p({\mathbb R}_+),L^{p'}({\mathbb R}_+)}.
$$
According to the implication ``$(ii)\Rightarrow (i)$"
of Theorem \ref{last} and Remark \ref{pre-dual}, $f\ast g$ is therefore a pointwise multiplier of $A_p(\Rdb_+)$, with
norm less than or equal to $\norm{f}_p\norm{g}_{p'}$.
We deduce that every $F\in A_p(\Rdb_+)$ is 
a pointwise multiplier of $A_p(\Rdb_+)$, with
norm less than or equal to $\norm{F}_{A_p}$.
This means that $A_p(\mathbb{R}_+)$ is a Banach algebra
for the pointwise product.
\end{rem}

\noindent
{\bf Acknowledgement.}
CL was supported by the ANR project 
{\it Noncommutative analysis on groups and quantum groups} 
(No./ANR-19-CE40-0002). SZ was supported by Projet I-SITE 
MultiStructure {\it Harmonic Analysis of noncommutative Fourier 
 and Schur multipliers over operator spaces}.
The authors
gratefully thank the Heilbronn Institute for Mathematical Research 
and the UKRI/EPSRC Additional Funding Program for Mathematical Sciences 
for the financial support through Heilbronn Grant R102688-101.

\bigskip

\end{document}